\documentclass[12pt]{amsart}

\usepackage{lineno,hyperref}
\modulolinenumbers[5]

\usepackage{enumerate}

\usepackage{geometry}                
\usepackage{graphicx}
\usepackage{amssymb}
\usepackage{float}
\usepackage{hyperref} 
\usepackage{bbm}
\usepackage{amsmath,amscd}
\usepackage{algorithm,algorithmic}
\usepackage[english]{babel}
\usepackage[latin1]{inputenc}
\usepackage{times}
\usepackage[T1]{fontenc}
\usepackage{latexsym}
\usepackage{graphics}
\usepackage{color}
\usepackage{mathrsfs}

\newtheorem{theorem}{Theorem}[section]
\newtheorem{proposition}[theorem]{Proposition}
\newtheorem{lemma}[theorem]{Lemma}
\newtheorem{corollary}[theorem]{Corollary}

\theoremstyle{definition}
\newtheorem{definition}[theorem]{Definition}
\newtheorem{example}[theorem]{Example}

\theoremstyle{remark}
\newtheorem{remark}[theorem]{Remark}

\numberwithin{equation}{section}

\newcommand{\bsi}{\boldsymbol{\sigma}}

\newcommand{\cE}{\mathcal{E}}

\newcommand{\cN}{\mathcal{N}}

\newcommand{\cL}{\mathcal{L}}

\newcommand{\cP}{\mathcal{P}}

\newcommand{\cK}{\mathcal{K}}

\newcommand{\bP}{\mathbb{P}}
\newcommand{\C}{\mathbb{C}}

\newcommand{\bE}{\mathbb{E}}

\newcommand{\bN}{\mathbb{N}}
\newcommand{\Z}{\mathbb{Z}}

\newcommand{\R}{\mathbb{R}}

\newcommand{\ep}{\epsilon }
\newcommand{\la}{\lambda }

\newcommand{\de}{\delta }

\newcommand{\si}{\sigma }

\newcommand{\ga}{\gamma }
\newcommand{\Ga}{\Gamma }

\newcommand{\Om}{\Omega }

\newcommand{\ones}{\mathbbm{1}}
\newcommand{\one}{\mathbf{1}}

\newcommand{\Diag}{\operatorname{Diag}}

\newcommand{\rea}{\operatorname{Re}}

\newcommand{\bd}[1]{\partial #1}

\newcommand{\Mod}{\operatorname{Mod}}

\newcommand{\MC}{\operatorname{MC}}

\newcommand{\defeq}{\mathrel{\mathop:}=}
\newcommand{\cCeff}{\mathop{\mathcal{C}_{\eff}}}
\newcommand{\cReff}{\mathcal{R}_{\eff}}

\newcommand{\eff}{\textrm{eff}}
\newcommand{\Adm}{\operatorname{Adm}}
\newcommand{\ext}{\operatorname{ext}}
\newcommand{\Dom}{\operatorname{Dom}}
\newcommand{\BL}{\operatorname{BL}}
\newcommand{\co}{\operatorname{co}}

\newcommand{\bi}{\begin{itemize}}
\newcommand{\ei}{\end{itemize}}

\newcommand{\vt}{\vspace{0.2cm}}

\title{Blocking duality for $p$-modulus on networks and applications} 
\author[Albin]{Nathan Albin$^1$}
\author[Clemens]{Jason  Clemens$^2$}
\author[Fernando]{Nethali Fernando$^1$}
\author[Poggi-Corradini]{Pietro Poggi-Corradini$^1$}

\thanks{Research supported by NSF n.~1201427 and n.~1515810}

\address{$^1$ Department of Mathematics, Kansas State University, Manhattan, KS}
\email{albin@math.ksu.edu} 
\email{tnethali@ksu.edu}
\email{pietro@math.ksu.edu}
\address{$^2$ Missouri Valley College, Marshall, MO}
\email{clemensj@moval.edu} 

\begin{document}


\begin{abstract}
  This paper explores the implications of blocking duality---pioneered
  by Fulkerson et al.---in the context of $p$-modulus on
  networks. Fulkerson's blocking duality is an analogue on networks to
  the method of conjugate families of curves in the plane.  The
  technique presented here leads to a general framework for studying
  families of objects on networks; each such family has a
  corresponding dual family whose $p$-modulus is essentially the
  reciprocal of the original family's.

  As an application, we give a modulus-based proof for the fact that
  effective resistance is a metric on graphs.  This proof immediately
  generalizes to yield a family of graph metrics, depending on the
  parameter $p$, that continuously interpolates among the
  shortest-path metric, the effective resistance metric, and the
  mincut ultrametric. In a second application, we establish a
  connection between Fulkerson's blocking duality and the
  probabilistic interpretation of modulus. This connection, in turn,
  provides a straightforward proof of several monotonicity properties
  of modulus that generalize known monotonicity properties of
  effective resistance.  Finally, we use this framework to expand on a
  result of Lov\'{a}sz in the context of randomly weighted graphs.
\end{abstract}


\maketitle

\section{Introduction}
Modulus on graphs (or networks) is a very flexible and general tool for measuring the richness of families of objects defined on a networks. For example, the underlying graphs can be directed or undirected, simple or multigraphs, weighted or unweighted. Also the objects that are being measured  can be very different. For instance, here are some flavors of modulus that  the first and last author have been studying:
\begin{itemize}
\item[-] {\it Connecting modulus.} This quantifies the richness of families of  walks connecting two given sets of vertices. By varying a parameter $p$, modulus generalizes classical quantities such as effective resistance (which only makes sense on undirected graphs), max flow/min cut, and shortest-path, see \cite{abppcw:ecgd2015}. Applications include new flexible centrality measures that have been used for modeling epidemic mitigation, see \cite{spcsa:jcam2016}. 
\item[-] {\it Loop modulus.} Looking at families cycles in a graph gives information about clustering and community detection, see \cite{spcas:PRE2017}.
\item[-] {\it Spanning tree modulus.} The modulus of the family of all spanning trees gives deep insights into the degree of connectedness of a network as well as exposing an interesting hierarchical structure, see   \cite{achpcst}.
\end{itemize}
The purpose of this paper is to develop the theory of Fulkerson blocking duality for modulus.
In Section \ref{sec:prelim} we recall the theory of modulus on networks. Then in Section \ref{sec:fulkerson}  and \ref{sec:examples} we develop the theory of Fulkerson duality for modulus. Also, in Section \ref{sec:prob}, we relate Fulkerson duality to Lagrangian duality and the probabilistic interpretation of modulus developed in \cite{abppcw:ecgd2015,apc,APCDSG}. 
Finally, we propose several applications of Fulkerson duality to demonstrate its
power and flexibility:
\begin{itemize}
\item  In section \ref{sec:metrics}, we give a new proof of the well-known fact that effective resistance is a metric on graphs, see for instance \cite[Corollary 10.8]{levin-peres-wilmer2009} for a proof based on commute times and \cite[Exercise 9.8]{levin-peres-wilmer2009} for one based on current flows. Assuming Fulkerson duality, our proof in Theorem \ref{thm:deltap} is very short and compelling. But it also has the added advantage of being the only proof we know that easily generalizes to a wider family of graph metrics based on modulus that continuously interpolate among the shortest-path metric, the effective resistance metric, and an ultrametric related to mincuts. None of the other classical proofs that effective resistance is a metric appear to generalize in this fashion. 
\item Furthermore, our proof in Theorem \ref{thm:deltap}, based on Fulkerson duality, allows us to establish the ``anti-snowflaking'' exponent for this family of graph metrics. Namely, we are able to find the exact largest exponent that each such metric can be raised to, while still being a metric on arbitrary graphs.
\item In Section \ref{sec:monotonicity}, we establish some useful monotonicity properties of modulus on a weighted graph $G=(V,E,\si)$ with respect to the edge-conductances $\si(e)$  (Theorem \ref{thm:monotonicity}).  Two of these properties generalize well-known facts about the behavior of resistor networks when a resistor's value is changed.  The Fulkerson blocker approach provides a third monotonicity property related to the expected edge usages of certain random objects on a graph.
\item Finally, in section \ref{sec:lovasz}, we use Fulkerson duality and the previously mentioned monotonicity property to study randomly weighted graphs. We first reinterpret and expand on some results of Lov\'{a}sz from \cite{lovasz:jcta2001}. We then establish a lower bound for the expected $p$-modulus of a family of objects in terms of modulus of the same family on the deterministic graph with edge weights given by their respective expected values (Theorem \ref{thm:bounds}).
\end{itemize}

\section{Preliminaries}\label{sec:prelim}
\subsection{Modulus in the Continuum}
\label{sec:mod-cont}

The theory of conformal modulus was originally developed in complex
analysis, see Ahlfors' comment on p.~81 of \cite{ahlfors1973}. The
more general theory of $p$-modulus grew out of the study of
quasiconformal maps, which generalize the notion of conformal maps to
higher dimensional real Euclidean spaces and, in fact, to abstract
metric measure spaces.  Intuitively, $p$-modulus provides a method for
quantifying the richness of a family of curves, in the sense that a
family with many short curves will have a larger modulus than a family
with fewer and longer curves.  The parameter $p$ tends to favor the
``many curves'' aspect when $p$ is close to $1$ and the ``short
curves'' aspect as $p$ becomes large. This phenomenon was explored
more precisely in \cite{abppcw:ecgd2015} in the context of networks.
The concept of discrete modulus on networks is not new, see for
instance
\cite{duffin:jmaa1962,schramm:ijm1993,haissinsky:fourier2009}. However,
recently the authors have started developing the theory of $p$-modulus
as a graph-theoretic quantity \cite{APCDSG,abppcw:ecgd2015}, with the
goal of finding applications, for instance to the study of epidemics
\cite{spcsa:jcam2016,GASSP}.

The concept of blocking duality explored in this paper is an analog of
the concept of conjugate families in the continuum.  As motivation for
the discrete theory to follow, then, let us recall the relevant
definitions from the continuum theory.  For now, it is convenient to
restrict attention to the $2$-modulus of curves in the plane, which,
as it happens, is a conformal invariant and thus has been carefully
studied in the literature.

Let $\Om$ be a domain in $\C$, and let $E,F$ be two continua in
$\overline{\Omega}$.  Define $\Ga=\Ga_{\Om}(E,F)$ to be the family of
all rectifiable curves connecting $E$ to $F$ in $\Omega$.  A
\emph{density} is a Borel measurable function
$\rho:\Om\rightarrow[0,\infty)$. We say that $\rho$ is
\emph{admissible} for $\Ga$ and write $\rho\in\Adm(\Ga)$, if
\begin{equation}
  \label{eq:cont-adm}
  \int_\ga \rho\;ds \geq 1\qquad\forall \ga\in\Ga.
\end{equation}
Now, we define the modulus of  $\Gamma$ as
\begin{equation}
\label{eq:cont-mod}
\Mod_2(\Ga) \defeq \inf_{\rho \in \Adm(\Gamma)} \int_{\Om} \rho^{2} dA.
\end{equation}
\begin{example}[The Rectangle]\label{ex:rectangle}
Consider a rectangle 
\[
\Om \defeq \{ z = x + i y \in \C : 0 < x < L, 0 < y < H \}
\]  
of height $H$ and length $L$.  Set
$E \defeq \{ z \in \overline{\Om} : \rea z = 0 \}$ and
$F \defeq \{ z \in \overline{\Om} : \rea z = L\}$ to be the leftmost
and rightmost vertical sides respectively. If
$\Gamma = \Gamma_{\Om}(E,F)$ then,
\begin{equation}\label{eq:rectangle}
\Mod_2(\Gamma) = \frac{H}{L}.
\end{equation}
To see this, assume $\rho \in \Adm(\Gamma)$. Then for all $0 < y < H$,
$\gamma_{y}(t) \defeq t + iy$ is a curve in $\Ga$, so
\[
\int_{\ga_y}\rho ds=\int_{0}^{L} \rho(t,y) dt \ge 1.
\]
Using the  Cauchy-Schwarz inequality we obtain,
\begin{align*}
  1 \le\left[ \int_{0}^{L} \rho(t,y)  dt \right]^{2} \le L \int_{0}^{L} \rho^{2}(t,y) dt.
\end{align*}
In particular, $L^{-1}\le \int_{0}^{L} \rho^{2}(t,y) dt$. Integrating
over $y$, we get
$$
\frac{H}{L}  \le \int_{\Om} \rho^{2} dA.
$$
So since $\rho$ was an arbitrary admissible density,
$\Mod_2(\Gamma) \ge \frac{H}{L}$.

In the other direction, define $\rho_{0}(z) = \frac{1}{L} \ones_{\Om}(z)$ and observe that $\int_{\Om} \rho_0^{2}dA = \frac{H L}{L^{2}} = \frac{H}{L}$. Hence, if we show that $\rho_0 \in \Adm(\Gamma)$, then $\Mod(\Gamma) \le \frac{H}{L}$. To see this note that for any $\ga\in \Ga$:
$$
\int_{0}^{L} \frac{1}{L} | \dot{\gamma}(t)| dt \ge \frac{1}{L} \int_{0}^{L} | \rea \dot{\gamma}(t) | dt \ge \frac{1}{L} \left( \rea \gamma(1) - \rea \gamma(0)  \right) \ge 1.
$$
This proves the formula (\ref{eq:rectangle}).
\end{example}
A famous and very useful result in this context is the notion of a
{\it conjugate family} of a connecting family. For instance, in the
case of the rectangle, the conjugate family $\Ga^*=\Ga^*_{\Om}(E,F)$
for $\Ga_{\Om}(E,F)$ consists of all curves that ``block'' or
intercept every curve $\ga\in\Ga_{\Om}(E,F)$. It's clear in this case
that $\Ga^*$ is also a connecting family, namely it includes every
curve connecting the two horizontal sides of $\Om$. In particular, by
(\ref{eq:rectangle}), we must have $\Mod_2(\Ga^*)=L/H$. So we deduce
that
\begin{equation}\label{eq:conjugate-duality}
\Mod_2(\Ga_{\Om}(E,F))\cdot \Mod_2(\Ga^*_{\Om}(E,F)) =1.
\end{equation}
One reason this reciprocal relation is useful is that upper-bounds for
modulus are fairly easy to obtain by choosing reasonable admissible
densities and computing their energy. However, lower-bounds are
typically harder to obtain. However, when an equation like
(\ref{eq:conjugate-duality}) holds, then upper-bounds for the modulus
of the conjugate family translate to lower-bounds for the given
family.

In higher dimensions, say in $\R^3$, the conjugate family of a
connecting family of curves consists of a family of surfaces, and
therefore one must consider the concept of \emph{surface modulus}, see
for instance \cite{rajala:gafa2005} and references therein.  It is
also possible to generalize the concept of modulus by replacing the
exponent $2$ in~\eqref{eq:cont-mod} with $p\ge 1$ and by replacing
$dA$ with a different measure.  

The principal aim of this paper is to establish a conjugate duality
formula similar to (\ref{eq:conjugate-duality}) for $p$-modulus on
networks, which we call {\it blocking duality}.

\subsection{Modulus on Networks}

A general framework for modulus of objects on networks was developed
in~\cite{apc}.  In what follows, $G= (V,E,\si)$ is taken to be a
finite graph with vertex set $V$ and edge set $E$. The graph may be
directed or undirected and need not be simple.  In general, we shall
assume a weighted graph with each edge assigned a corresponding weight
$0<\si(e)<\infty$.  When we refer to an unweighted graph, we shall
mean a graph for which all weights are assumed equal to one.

The theory in~\cite{apc} applies to any finite family of ``objects''
$\Ga$ for which each $\ga\in\Ga$ can be assigned an associated
function $\cN(\ga,\cdot): E\rightarrow \R_{\ge 0}$ that measures the
{\it usage of edge $e$ by $\ga$}.  Notationally, it is convenient to
consider $\cN(\ga,\cdot)$ as a row vector
$\cN(\ga,\cdot)\in\R_{\ge 0}^E$, indexed by $e\in E$.  In order to
avoid pathologies, it is useful to assume that $\Ga$ is non-empty and that each $\gamma\in\Gamma$
has positive usage on at least one edge. When this is the case, we will say that $\Ga$ is {\it non-trivial}.
In the following it will be useful to define the quantity:
\begin{equation}\label{eq:usagelb}
\cN_{\rm min}:=\min_{\ga\in\Ga}\min_{e: \cN(\ga,e)\neq 0} \cN(\ga,e).
\end{equation}
Note that, for $\Ga$ non-trivial, $\cN_{\rm min}>0$.

Some examples of objects and
their associated usage functions are the following.
\begin{itemize}
\item To a walk $\ga=x_0\ e_1\ x_1\ \cdots\ e_n\ x_n$ we can
  associate the traversal-counting function $\cN(\ga,e)=$ number times
  $\ga$ traverses $e$. In this case $\cN(\ga,\cdot)\in\Z_{\ge 0}^E$.
  \vt
\item To each subset of edges $T\subset E$ we can associate the
  characteristic function $\cN(T,e)=\ones_T(e)=1$ if $e\in T$ and $0$
  otherwise.  Here, $\cN(\ga,\cdot)\in\{0,1\}^E$.  \vt
\item To each flow $f$ we can associate the volume function
  $\cN(f,e)=|f(e)|$. Therefore, $\cN(\ga,\cdot)\in \R_{\ge 0}^E$.
\end{itemize}

As a function of two variables, the function $\cN$ can be thought of
as a matrix in $\mathbb{R}^{\Gamma\times E}$, indexed by pairs
$(\gamma,e)$ with $\gamma$ an object in $\Ga$ and $e$ an edge in $E$.
This matrix $\cN$ is called the \emph{usage matrix} for the family $\Ga$. Each row
of $\cN$ corresponds to an object $\ga\in\Ga$ and records the usage of
edge $e$ by $\ga$. At times will write $\cN(\Ga)$ instead of $\cN$, to avoid ambiguity. Note, that the families $\Ga$ under consideration
may very well be infinite (e.g.~families of walks), so $\cN$ may have
infinitely many rows.  For this paper, we shall assume $\Ga$ is
finite. 

This assumption is not quite as restrictive as it might seem.
In~\cite{APCDSG} it was shown that any family $\Gamma$ with an
integer-valued $\cN$ can be replaced, without changing the modulus, by
a finite subfamily.  For example, if $\Gamma$ is the set of all walks
between two distinct vertices, the modulus can be computed by
considering only simple paths.  This result implies a similar
finiteness result for any family $\Gamma$ whose usage matrix $\cN$ is
rational with positive entries bounded away from zero.

By analogy to the continuous setting, we define a \emph{density} on
$G$ to be a nonnegative function on the edge set:
$\rho:E\to[0,\infty)$.  The value $\rho(e)$ can be thought of as the
{\it cost of using edge $e$}.  It is notationally useful to think of
such functions as column vectors in $\R_{\ge 0}^E$.  In order to
mimic~\eqref{eq:cont-adm}, we define for an object $\ga\in\Ga$
\[
\ell_\rho(\ga):=\sum_{e\in E} \cN(\ga,e)\rho(e) = (\cN \rho)(\ga),
\]
representing the {\it total usage cost} for $\ga$ with the given edge
costs $\rho$.  In linear algebra notation, $\ell_\rho(\cdot)$ is the
column vector resulting from the matrix-vector product $\cN\rho$.  As
in the continuum case, then, a density $\rho\in\R_{\ge 0}^E$ is called
{\it admissible for $\Ga$}, if
\[
\ell_\rho(\ga) \geq 1\qquad\forall \ga\in\Ga;
\quad
\text{or equivalently, if}
\quad
\ell_\rho(\Ga)\defeq\inf_{\ga\in\Ga}\ell_\rho(\ga) \geq 1.
\]
In matrix notation, $\rho$ is admissible if
\[
\cN \rho \geq {\one},
\]
where $\one$ is the column vector of ones and the inequality is
understood to hold elementwise.  By analogy, we define the set
\begin{equation}\label{eq:Adm}
\Adm(\Ga)=\left\{\rho\in\R_{\ge 0}^E: \cN \rho\geq 1\right\}
\end{equation}
to be the set of admissible densities.

Now, given an exponent $p\ge 1$ we define the \emph{$p$-energy} on
densities, corresponding to the area integral from the continuum case,
as
\begin{equation*}
  \cE_{p,\si}(\rho) \defeq \sum_{e\in E} \si(e)\rho(e)^p,
\end{equation*}
with the weights $\si$ playing the role of the area element $dA$.  In
the unweighted case ($\sigma\equiv 1$), we shall use the notation
$\cE_{p,1}$ for the energy. For $p=\infty$, we also define the
unweighted and weighted {\it $\infty$-energy} respectively as
\begin{equation*}
  \cE_{\infty,1}(\rho) \defeq
  \lim_{p\to\infty}\left(\cE_{p,\si}(\rho)\right)^{\frac{1}{p}} =
  \max_{e\in E}\rho(e)
\end{equation*}
and
\begin{equation*}
  \cE_{\infty,\sigma}(\rho) \defeq
  \lim_{p\to\infty}\left(\cE_{p,\si^p}(\rho)\right)^{\frac{1}{p}} =
  \max_{e\in E}\sigma(e)\rho(e)  
\end{equation*}
This leads to the following definition.

\begin{definition}
  Given a graph $G= (V,E,\si)$, a family of objects $\Ga$ with usage
  matrix $\cN\in\mathbb{R}^{\Ga\times E}$, and an exponent
  $1\le p\le\infty$, the {\it $p$-modulus} of $\Gamma$ is
\[ \Mod_{p,\si}(\Ga)\defeq \inf_{\rho\in \Adm(\Ga)}\cE_{p,\si}(\rho)\]
\end{definition}
Equivalently, $p$-modulus corresponds to the following optimization problem
\begin{equation}
  \label{eq:mod-as-cvx}
  \begin{split}
    \text{minimize} &\qquad \cE_{p,\si}(\rho) \\
    \text{subject to} &\qquad  \rho\ge 0,\quad \cN \rho \geq 1
  \end{split}
\end{equation}
where each object $\ga\in\Ga$ determines one inequality constraint.

\begin{remark}\label{rem:remarks}
  \begin{itemize}
  \item[(a)] When $\rho_0\equiv 1$, we drop the subscript and
    write $\ell(\ga)\defeq\ell_{\rho_0}(\ga)$. If $\ga$ is a walk,
    then $\ell(\ga)$ simply counts the number of hops that the walk
    $\ga$ makes.
  \item[(b)] For $1<p<\infty$ a unique extremal density $\rho^*$
    always exists and satisfies
    $0\leq \rho^*\leq \cN_{\text{min}}^{-1}$, where $\cN_{\text{min}}$
    is defined in (\ref{eq:usagelb}). Existence and uniqueness
    follows by compactness and strict convexity of $\cE_{p,\si}$, see
    also Lemma 2.2 of \cite{abppcw:ecgd2015}.  The upper bound on
    $\rho^*$ follows from the fact that each row of $\cN$ contains at
    least one nonzero entry, which must be at least as large as
    $\cN_{\text{min}}$.  In the special case when $\cN$ is integer
    valued, the upper bound can be taken to be $1$.
  \end{itemize}
\end{remark}
The next result shows that modulus is a ``capacity'', in the mathematical sense, on families of objects. This is a known fact, see \cite[Prop. 3.4]{APCDSG} for the case of families of walks. We reproduce a proof here for completeness.
\begin{proposition}[Basic Properties]\label{prop:basicprop}
Let $G=(V,E,\si)$ be a simple finite graph with edge-weights $\si\in\R_{>0}^E$. For simplicity, all families of objects on $G$ are assumed to be non-trivial. Then, for $p\in[1,\infty]$, the following hold:
\begin{itemize}
\item[(a)] {\bf  Monotonicity:} Suppose $\Ga$ and $\Ga'$ are families of objects on $G$ such 
that $\Ga\subset\Ga'$, meaning that the matrix $\cN(\Ga)$ is the
    restriction of the matrix $\cN(\Ga')$ to the rows from $\Ga$. Then,  
\begin{equation}\label{eq:monotonicity}
\Mod_{p,\si}(\Ga)\leq \Mod_{p,\si}(\Ga').
\end{equation}
\item[(b)] {\bf Countable Subadditivity:} Suppose $1\le p<\infty$, and let $\{\Ga_j\}_{j=1}^\infty$ be a sequence of families of objects on $G$. then
\begin{equation}\label{eq:subadditivity}
\Mod_{p,\si}\left(\bigcup_{j=1}^\infty \Ga_j\right)\le \sum_{j=1}^\infty \Mod_{p,\si}(\Ga_j).
\end{equation}
\end{itemize}
\end{proposition}
\begin{proof}
For monotonicity, note that  $\Adm(\Ga')\subset \Adm(\Ga)$.

For subadditivity, we first fix $p\in [1,\infty)$. Let $\Gamma \defeq \bigcup_{j=1}^{\infty} \Gamma_j$. For each $j$, choose $\rho_{j} \in \Adm(\Gamma_j)$ such that 
\[
\cE_{p,\si}(\rho_j) = \Mod_{p,\si}\left( \Gamma_j \right).
\]
 Assuming that the right-hand side of~\eqref{eq:subadditivity} is finite, then, since $\si>0$ and $\rho_j\ge 0$,
 \begin{equation*}
   \sum_{e\in E}\si(e)\sum_{j=1}^\infty \rho_j(e)^p = \sum_{j=1}^\infty \sum_{e\in E}\si(e)\rho_j(e)^p
   = \sum_{j=1}^\infty \Mod_{p,\si}(\Ga_j) < \infty.
 \end{equation*}
So, $\rho \defeq \left( \sum_{j=1}^{\infty} \rho_{j}^{p}  \right)^{\frac{1}{p}}$ is also finite. For any $\gamma \in \Gamma$, there exists $k \in \bN$ so that $\gamma \in \Gamma_{k}$. In particular, since $\rho \ge \rho_{k}$, we have $\ell_{\rho}(\gamma) \ge 1$. This shows that $\rho \in \Adm(\Gamma)$. Moreover,
\begin{align*}
\Mod_{p,\si} \Gamma &\le \cE_{p,\si}(\rho) = \sum_{e \in E}\si(e) \rho(e)^{p} = \sum_{e \in E}\si(e) \sum_{j=1}^{\infty} \rho_{j}(e)^{p} = \sum_{j=1}^{\infty} \sum_{e \in E}\si(e) \rho_{j}(e)^{p} \\
&= \sum_{j=1}^{\infty} \cE_{p,\si}(\rho_{j}) =\sum_{j=1}^{\infty} \Mod_{p,\si}(\Gamma_j).
\end{align*}

We leave the case $p=\infty$ to the reader (one can even replace the sum with max).
\end{proof}

\subsection{Connection to classical quantities}

The concept of $p$-modulus generalizes known classical ways of
measuring the richness of a family of walks.  Let $G=(V,E)$ 
and two vertices $a$ and $b$ in $V$ be given. We define the {\it connecting
  family} $\Ga(a,b)$ to be the family of all simple paths in $G$ that
start at $a$ and end at $b$.  To this family, we assign the usage
function $\cN(\gamma,e)$ to be $1$ when $e\in\gamma$ and $0$
otherwise. Classically, there are three main ways to measure the richness of $\Ga(a,b)$.
\begin{itemize}
\item  {\it Mincut:}
A subset $S\subset V$ is called a $ab$-{\it cut} if $a\in S$ and $b\not\in S$.
To every $ab$-cut $S$ we assign the edge usage
$\cN(S,e)=1$ for every $e=\{x,y\}\in E$ such that $x\in S$ and $y\not\in S$; and $\cN(S,e)=0$ otherwise.  The support of $\cN(S,\cdot)$ is also known as the {\it edge-boundary} $\bd S$. Given edge-weights $\si$, the {\it size} of an $ab$-cut is measured by
$|\bd S|:=\sum_{e\in E}\si(e)\cN(S,e)$. 
We define the {\it min cut} between $a$ and $b$ to:
\begin{equation*}
  \MC(a,b) := \min\left\{|\bd S|:S\text{ is an $ab$-cut}\right\}.
\end{equation*}
\item {\it Effective Resistance:} When $G$ is undirected, it can be thought of as an electrical network with edge conductances given by the weights $\si$,
see~\cite{doyle-snell1984}. Then {\it effective resistance} $\cReff(a,b)$ is the voltage drop necessary to pass $1$ Amp of current between $a$ and $b$ through $G$
\cite{doyle-snell1984}. In this case, given two vertices $a$ and
$b$ in $V$, we write $\cCeff(a,b):=\cReff(a,b)^{-1}$ for the {\rm effective conductance}
between $a$ and $b$. 
\item {\it Shortest-path:} Finally, the (unweighted) {\it shortest-path distance} between $a$ and $b$ refers to the length of the shortest path from $a$ to $b$, where the length of a path $\ga$ is 
$\ell(\ga):=\sum_{e\in E}\cN(\ga,e),$
and we write \[\ell(\Ga):=\inf_{\ga\in \Ga}\ell(\ga)\] for the shortest length of a family $\Ga$.
\end{itemize}
The following result is a slight modification of
the results in \cite[Section~5]{abppcw:ecgd2015}, taking into account
the definition of $\cN_{\text{min}}$ in (\ref{eq:usagelb}).

\begin{theorem}[\cite{abppcw:ecgd2015}]\label{thm:generalize}
  Let $G=(V,E,\si)$ be a graph with edge weights $\sigma$. Let $\Ga$ be a
  nontrivial family of objects on $G$ with usage matrix $\cN$ and let
  $\sigma(E) := \sum_{e\in E}\sigma(e)$. Then the function
  $ p\mapsto \Mod_{p,\si}(\Ga)$ is continuous for
  $1\leq p< \infty$, and the following two monotonicity properties hold
  for $1\le p \le p' <\infty$.
  \begin{align}
    \label{eq:monotone-decr}
    \cN_{\text{min}}^p \Mod_{p,\si}(\Ga) &\ge \cN_{\text{min}}^{p'} \Mod_{p',\si}(\Ga),\\
    \label{eq:monotone-incr}
    \left(\si(E)^{-1}\Mod_{p,\si}(\Ga)\right)^{1/p} &\le
    \left(\si(E)^{-1}\Mod_{p',\si}(\Ga)\right)^{1/p'}.
  \end{align}
 Moreover, let $a\neq b$ in $V$ be given and set $\Ga$ equal to the connecting family $\Ga(a,b)$. Then,
{\renewcommand\arraystretch{1.8} 
\[
\begin{array}{lll}
\bullet\ \text{For $p=1$,} & 
\Mod_{1,\sigma}(\Ga)=\min\{|\bd S|:\text{\rm $S$ an $ab$-cut}\} = \MC(a,b) &
\text{Min cut.}\\
\bullet\ \text{For $p=2$,}&
\Mod_{2,\sigma}(\Ga)=\cCeff(a,b) = \cReff(a,b)^{-1} &
\text{Effective conductance.}\\
\bullet\  \text{For $p=\infty$,} &
\Mod_{\infty,1}(\Ga)=\lim\limits_{p\to\infty}\Mod_{p,\si}(\Gamma)^{\frac{1}{p}} =\ell(\Ga)^{-1} & \text{Reciprocal of shortest-path.}\\
\end{array}
\]
}
\end{theorem}
\begin{remark}
  An early version of the case $p=2$ is due to Duffin
  \cite{duffin:jmaa1962}. The proof in \cite{abppcw:ecgd2015} was
  guided by a very general result in metric
  spaces~\cite{heinonen2001}.
\end{remark}

The theorem stated  in \cite[Section~5]{abppcw:ecgd2015} does not hold in this context verbatim, but can be easily adapted.  The only issue to take care of is the value of
$\cN_{\text{min}}$.
Since the previous paper dealt only with families of walks, $\cN$ was
integer valued and, thus, $\cN_{\text{min}}$ could be assumed no
smaller than $1$.  This gave rise to an inequality of the form
$0\le\rho^*\le 1$ that was used to establish a monotonicity property.
When $\cN$ is not restricted to integer values, the bound on $\rho^*$
should be replaced by $0\le\rho^*\le\cN_{\text{min}}^{-1}$ (see
Remark~\ref{rem:remarks} (c)).  Repeating the proof
of~\cite[Thm.~5.2]{abppcw:ecgd2015} with the corrected upper bound and
rephrasing in the current context yields the following theorem.

\begin{example}[Basic Example]\label{ex:basic}
  Let $G$ be a graph consisting of $k$ simple paths in parallel, each
  path taking $\ell$ hops to connect a given vertex $s$ to a given
  vertex $t$. Assume also that $G$ is unweighted, that is
  $\si\equiv 1$. Let $\Ga$ be the family consisting of the $k$ simple
  paths from $s$ to $t$.  Then $\ell(\Ga)=\ell$ and the size of the
  minimum cut is $k$. A straightforward computation shows that
\[
\Mod_p(\Ga)=\frac{k}{\ell^{p-1}}\quad\mbox{for } 1\le p<\infty,\qquad
\Mod_{\infty,1}(\Ga)=\frac{1}{\ell}.
\]
In particular, $\Mod_p(\Ga)$ is continuous in $p$, and
$\lim_{p\to\infty}\Mod_p(\Ga)^{1/p}=\Mod_{\infty,1}(\Ga)$.  Intuitively,
when $p\approx 1$, $\Mod_p(\Ga)$ is more sensitive to the number of
parallel paths, while for $p\gg 1$, $\Mod_p(\Ga)$ is more sensitive to
short walks.
\end{example}

\subsection{Lagrangian Duality and the Probabilistic Interpretation}
\label{sec:dual-prob}

The optimization problem~\eqref{eq:mod-as-cvx} is an ordinary convex
program, in the sense of~\cite[Sec.~28]{Rockafellar1970}.  Existence
of a minimizer follows from compactness, and uniqueness holds when
$1<p<\infty$ by strict convexity of the objective function.
Furthermore, it can be shown that strong duality holds in the sense
that a maximizer of the Lagrangian dual problem exists and has dual
energy equal to the modulus.  The Lagrangian dual problem was derived
in detail in~\cite{abppcw:ecgd2015}.  The Lagrangian dual was later
reinterpreted in a probabilistic setting in~\cite{apc}.  

In order to formulate the probabilistic dual, we let $\cP(\Gamma)$
represent the set of probability mass functions (pmfs) on the set
$\Gamma$.  In other words, $\cP(\Gamma)$ contains the set of vectors
$\mu\in\mathbb{R}_{\ge 0}^\Gamma$ with the property that
$\mu^T\one = 1$.  Given such a $\mu$, we can define a
$\Gamma$-valued random variable $\underline{\gamma}$ with distribution
given by $\mu$:
$\mathbb{P}_\mu\left(\underline{\gamma}=\gamma\right) = \mu(\gamma)$.
Given an edge $e\in E$, the value $\cN(\underline{\gamma},e)$ is again
a random variable, and we represent its expectation (depending on the
pmf $\mu$) as $\mathbb{E}_\mu\left[\cN(\underline{\gamma},e)\right]$.
The probabilistic interpretation of the Lagrangian dual can now be
stated as follows.

\begin{theorem}
  \label{thm:prob-interp}
  Let $G=(V,E)$ be a finite graph with edge weights $\sigma$, and let
  $\Gamma$ be a non-trivial finite family of objects on $G$ with usage matrix
  $\cN$.  Then, for any $1<p<\infty$, letting $q:=p/(p-1)$ be the conjugate exponent to $p$, we have
  \begin{equation}
    \label{eq:prob-interp}
    \Mod_{p,\sigma}(\Gamma)^{-\frac{1}{p}} = 
    \left(
      \min_{\mu\in\cP(\Gamma)}\sum_{e\in E}\sigma(e)^{-\frac{q}{p}}
      \bE_\mu\left[\cN(\underline{\gamma},e)\right]^q
    \right)^{\frac{1}{q}}.
  \end{equation}
  Moreover, any optimal measure $\mu^*$, must satisfy
  \[
  \bE_{\mu^*}\left[\cN(\underline{\gamma},e)\right]=\frac{\si(e)\rho^*(e)^{\frac{p}{q}}}{\Mod_{p,\si}(\Ga)}\qquad\forall
e\in E,
  \]
  where $\rho^*$ is the unique extremal density for $\Mod_{p,\si}(\Ga)$.
\end{theorem}
Theorem \ref{thm:prob-interp} is a consequence of the theory developed in~\cite{apc}.  However, since it was only remarked on in~\cite{apc}, we provide a detailed proof here.
\begin{proof}
  The optimization problem~\eqref{eq:mod-as-cvx} is a standard convex
  optimization problem.  Its Lagrangian dual problem, derived
  in~\cite{abppcw:ecgd2015}, is
  \begin{equation}\label{eq:dual-p}
    \begin{split}
      \text{maximize}\quad&
      \sum_{\gamma\in\Gamma}\lambda(\gamma) - (p-1)\sum_{e\in E}\sigma(e)\left(
        \frac{1}{p\sigma(e)}\sum_{\gamma\in\Gamma}\cN(\gamma,e)\lambda(\gamma)
      \right)^{\frac{p}{p-1}}\\
      \text{subject to}\quad& \lambda(\gamma)\ge 0\quad\forall\gamma\in\Gamma.
    \end{split}
  \end{equation}
  It can be readily verified that strong duality holds (i.e., that the
  minimum in~\eqref{eq:mod-as-cvx} equals the maximum
  in~\eqref{eq:dual-p}) and that both extrema are attained.  Moreover,
  if $\rho^*$ is the unique minimizer of the modulus problem and
  $\lambda^*$ is any maximizer of the Lagrangian dual, then the
  optimality conditions imply that
  \begin{equation}\label{eq:rho-lambda}
    \rho^*(e) = \left(\frac{1}{p\sigma(e)}\sum_{\gamma\in\Gamma}
      \cN(\gamma,e)\lambda^*(\gamma)\right)^{\frac{1}{p-1}}.
  \end{equation}
  By decomposing $\lambda\in\mathbb{R}_{\ge 0}^\Gamma$ as
  $\lambda=\nu\mu$ with $\nu\ge 0$ and $\mu\in\cP(\Gamma)$, we can
  rewrite~\eqref{eq:dual-p} as
  \begin{equation*}
    \max_{\nu\ge 0}\left\{
      \nu - (p-1)\left(\frac{\nu}{p}\right)^q\min_{\mu\in\cP(\Gamma)}
      \sum_{e\in E}\sigma(e)^{-\frac{q}{p}}
      \left(
        \sum_{\gamma\in\Gamma}\cN(\gamma,e)\mu(\gamma)
      \right)^q
    \right\}.
  \end{equation*}
  The minimum over $\mu$ can be recognized as the minimum
  in~\eqref{eq:prob-interp}.  Let $\alpha$ be its minimum value.  Then
  the maximum over $\nu\ge 0$ is attained at
  $\nu^* := p\alpha^{-\frac{p}{q}}$, and strong duality implies that
  \begin{equation*}
    \Mod_{p,\sigma}(\Gamma) = \nu^*-(p-1)\left(\frac{\nu^*}{p}\right)^q\alpha
    = \alpha^{-\frac{p}{q}}.
  \end{equation*}
  Thus,
  \begin{equation*}
    \min_{\mu\in\cP(\Gamma)}\sum_{e\in E}\sigma(e)^{-\frac{q}{p}}
    \bE_\mu\left[\cN(\underline{\gamma},e)\right]^q
    = \alpha = \Mod_{p,\sigma}(\Gamma)^{-\frac{q}{p}},
  \end{equation*}
  proving~\eqref{eq:prob-interp}.  The remainder of the theorem
  follows from~\eqref{eq:rho-lambda}:
  \begin{equation*}
    \rho^*(e) = \left(\frac{\nu^*}{p\sigma(e)}\sum_{\gamma\in\Gamma}
      \cN(\gamma,e)\mu^*(\gamma)\right)^{\frac{1}{p-1}}
    = \alpha^{-1}\sigma(e)^{-\frac{q}{p}}
    \bE_{\mu^*}\left[\cN(\underline{\gamma},e)\right]^{\frac{q}{p}}
  \end{equation*}
\end{proof}
\begin{remark}
The probabilistic interpretation is particularly informative when
$p=2$, $\sigma\equiv 1$, and $\Gamma$ is a
collection of subsets of $E$, so that $\cN$ is a $(0,1)$-matrix defined as
$\cN(\gamma,e)=\ones_\gamma(e)$.
In this case, this duality relation can be
expressed as
\begin{equation*}
  \Mod_2(\Gamma)^{-1} = \min_{\mu\in\cP(\Gamma)}\mathbb{E}_\mu\left|
    \underline{\gamma}\cap\underline{\gamma}'
  \right|,
\end{equation*}
where $\underline{\gamma}$ and $\underline{\gamma}'$ are two
independent random variables chosen according to the pmf $\mu$, and
$\left| \underline{\gamma}\cap\underline{\gamma}' \right|$ is their
\emph{overlap} (also a random variable).  In other words, computing
the $2$-modulus in this setting is equivalent to finding a pmf that
minimizes the expected overlap of two iid $\Gamma$-valued random
variables.
\end{remark}
In the present work, we are interested in a different but closely
related duality called \emph{blocking duality}.

\section{Blocking Duality and $p$-Modulus}\label{sec:fulkerson}

In this section, we introduce blocking duality for modulus.  If $\Ga$ is a finite non-trivial family of objects on a graph $G$, the admissible set $\Adm(\Ga)$, defined in~\eqref{eq:Adm}, is determined by finitely many inequalities:
\begin{equation*}
  \sum_{e\in E}\cN(\ga,e)\rho(e) \ge 1\quad\forall\ga\in\Ga.
\end{equation*}
Thus, it is possible to identify $\Ga$ with the rows of its edge usage matrix $\cN$ or, equivalently, with the corresponding points in $\R^E_{\ge 0}$.

\subsection{Fulkerson's theorem}

First, we recall some general definitions. Let $\cK$ be the set of all closed convex sets $K\subset\R_{\ge 0}^E$ that are {\it recessive}, in the sense that $K+\mathbb{R}_{\ge 0}^E=K$.  To
avoid trivial cases, we shall assume that
$\varnothing\subsetneq K\subsetneq\mathbb{R}_{\ge 0}^E$, for $K\in\cK$. 
\begin{definition}
For each $K\in\cK$ there is an associated \emph{blocking
  polyhedron}, or \emph{blocker},
\begin{equation*}
  \BL(K) := \left\{\eta\in\mathbb{R}_{\ge 0}^E : 
    \eta^T\rho\ge 1,\;\;\forall\rho\in K \right\}.
\end{equation*}
\end{definition}

\begin{definition}
Given $K\in\cK$ and a point $x\in K$ we say that $x$ is an {\it extreme point} of $K$ if
$x=tx_1+(1-t)x_2$ for some $x_1,x_2\in K$ and some $t\in (0,1)$, implies that $x_1=x_2=x$.
Moreover, we let ${\rm ext}(K)$  be the set of all extreme points of $K$ .
\end{definition}

\begin{definition} 
The {\it dominant} of a set $P\subset\R_{\ge 0}^E$ is the recessive closed convex set
\[
\Dom(P)=\co(P)+\R_{\ge 0}^E,
\]
where $\co(P)$ is the convex hull of $P$.
\end{definition}

When $\Ga$ is finite, $\Adm(\Ga)$ has finitely many faces. However, $\Adm(\Ga)$ is also determined by its finitely many extreme points, or ``vertices'' in $\R^{E}_{\ge 0}$. 
In fact, since $\Adm(\Ga)$ is a recessive closed convex set, it equals the dominant of its extreme points ${\rm ext}(\Adm(\Ga))$, see \cite[Theorem~18.5]{Rockafellar1970}.  In the present notations,
\begin{equation}\label{eq:caratheodory}
  \Adm(\Ga) = \Dom(\ext(\Adm(\Ga))).
\end{equation}

\begin{definition}\label{def:fulkerson}
Suppose $G=(V,E)$ is a finite graph and $\Ga$ is a finite non-trivial family of objects on $G$. We say that the family
$$\hat{\Ga}:={\rm ext}(\Adm(\Ga))=\{\hat{\ga}_1,\dots,\hat{\ga}_s\}\subset\R_{\ge 0}^E,$$ consisting  of the extreme points of $\Adm(\Ga)$, is the {\it  Fulkerson blocker} of $\Ga$. We define the matrix
$\hat{\cN}\in\mathbb{R}_{\ge 0}^{\hat{\Ga}\times E}$ to be the matrix whose rows
are the vectors $\hat{\ga}^T$, for $\hat{\ga}\in\hat{\Ga}$. 
\end{definition}

\begin{theorem}[Fulkerson~\cite{Fulkerson1968}]
  \label{thm:fulkerson}
Let $G=(V,E)$ be a graph and let $\Ga$ be a non-trivial finite family of objects on $G$. 
Let $\hat{\Ga}$ be the Fulkerson blocker of $\Ga$. Then
\begin{itemize}
\item[(1)] $\Adm(\Ga)=\Dom(\hat{\Ga})=\BL(\Adm(\hat{\Ga}))$;
\item[(2)] $\Adm(\hat{\Ga})=\Dom(\Ga)=\BL(\Adm(\Ga))$;
\item[(3)] $\hat{\hat{\Ga}}\subset \Ga$. 
\end{itemize}
\end{theorem}
In words, (3) says that the extreme points of $\Adm(\hat{\Ga})$ are a subset of $\Ga$. 
Combining  (1) and (2) we get the following relationships in terms of $\Ga$ alone.
\begin{corollary}\label{cor:admbldom}
Let $G=(V,E)$ be a graph and let $\Ga$ be a nontrivial finite family of objects on $G$. Then,
\[
\BL(\BL(\Adm(\Ga)))=\Adm(\Ga)\qquad\text{and}\qquad \BL(\BL(\Dom(\Ga)))=\Dom(\Ga).
\]
as well as
\[
\Adm(\Ga)=\BL\left(\Dom(\Ga)\right)\qquad\text{and}\qquad \BL(\Adm(\Ga))=\Dom(\Ga).
\]
\end{corollary}
We include a proof of Theorem \ref{thm:fulkerson} for the reader's convenience.
\begin{proof}
We first prove (2).
Suppose $\eta\in\BL(\Adm(\Ga))$. Then $\eta^T\rho\ge 1$, for every $\rho\in\Adm(\Ga)$. In particular, since every row of $\hat{\cN}$ is an extreme point of $\Adm(\Ga)$, we have
\begin{equation}\label{eq:vertices}
\hat{\cN}\eta\ge 1.
\end{equation}
In other words, $\eta \in\Adm(\hat{\Ga})$.
Conversely, suppose $\eta \in\Adm(\hat{\Ga})$, that is (\ref{eq:vertices}) holds. Since
\[
\Adm(\Ga)={\rm co}(\hat{\Ga})+\R_{\ge 0}^E,
\]
for every $\rho\in\Adm(\Ga)$,
there is a probability measure $\nu\in\cP(\hat{\Ga})$ and a vector $z\ge 0$ such that
\[
\rho = \hat{\cN}^T\nu + z
\]
And by (\ref{eq:vertices}),
\[
\eta^T\rho=\eta^T\hat{\cN}^T\nu + \eta^Tz
\ge \nu^T1 + \eta^Tz \ge 1.
\]
So $\eta\in\BL(\Adm(\Ga))$.

Note that $\eta\in\BL(\Adm(\Ga))$ if and only if the value of the following linear program is greater or equal $1$.
\begin{equation}
  \label{eq:ahat}
  \begin{split}
    \text{minimize}\qquad &\eta^T\rho\\
    \text{subject to}\qquad & \cN \rho\geq \one,\ \rho\ge 0,
  \end{split}
\end{equation}
where $\cN$ is the usage matrix for $\Ga$.
The Lagrangian for this problem is
\[
\cL(\rho,\la,t):= \eta^T\rho +\la^T(\one-\cN\la)-t^T\rho=\la^T\one+\rho^T(\eta-\cN^T\la-t),
\]
with $\rho\in\R^E$, $\la\in\R_{\ge 0}^\Ga$ and $t\in \R_{\ge 0}^E$. In particular, the dual problem is
\begin{equation}
  \label{eq:ahat-dual}
  \begin{split}
    \text{maximize}\qquad &\la^T\one\\
    \text{subject to}\qquad & \cN^T \la\leq \eta,\ \la\ge 0.
  \end{split}
\end{equation}
Splitting $\la=s\nu$, with $s\ge 0$ and $\nu\in\cP(\Ga)$, we can rewrite this problem as
\begin{equation}
  \label{eq:ahat-dual-prob}
  \begin{split}
    \text{maximize}\qquad &s\\
    \text{subject to}\qquad & s\cN^T \nu\leq \eta,\  \nu\in\cP(\Ga).
  \end{split}
\end{equation}
By strong duality, $\eta\in\BL(\Adm(\Ga))$ if and only if there is $s\ge 1$ and $\nu\in\cP(\Ga)$ so that 
\[
\eta\ge s\cN^T\nu.
\]
Namely, $\eta\in\BL(\Adm(\Ga))$ implies that $\eta\ge\cN^T\nu$,
so $\eta\in \Dom(\Ga)$. 

Conversely, if $\eta\in\Dom(\Ga)$, then there is a $\nu\in\cP(\Ga)$ such that $\eta\ge\cN^T\nu$. So we have proved (2).
In particular, since $\hat{\hat{\Ga}}$ is the set of extreme points of $\Adm(\hat{\Ga})$ by Definition~\ref{def:fulkerson}, it follows from (2) that
\begin{equation*}
  \hat{\hat{\Ga}} = \ext(\Adm(\hat{\Ga})) = \ext(\Dom(\Ga)).
\end{equation*}
Since any extreme point of $\Dom(\Ga)$ must be present in $\Ga$, we conclude that $\hat{\hat{\Ga}}\subset\Ga$, and hence (3) is proved as well.

To prove (1), we apply (2) to $\hat{\Ga}$ and find that
\[
\BL(\Adm(\hat{\Ga}))=\Adm(\hat{\hat{\Ga}})\supset \Adm(\Ga),
\]
where the last inclusion follows from (3), since $\hat{\hat{\Ga}}\subset \Ga$.  Also, by (3) applied to $\hat{\Ga}$, the extreme points of $\Adm(\hat{\hat{\Ga}})$ are a subset of $\hat{\Ga}$ and therefore they are a subset  of $\ext(\Adm(\Ga))$. This implies that $\Adm(\hat{\hat{\Ga}})\subset\Adm(\Ga)$.
So we have $\BL(\Adm(\hat{\Ga}))= \Adm(\Ga)$.

Moreover, by (2) applied to $\hat{\Ga}$, we get that
\[
\BL(\Adm(\hat{\Ga}))=\Dom(\hat{\Ga}).
\]
So (1) is proved as well.
\end{proof}

\subsection{Blocking duality for $p$-modulus}

\begin{theorem}
  \label{thm:duality}
  Let $G=(V,E)$ be a graph and let $\Ga$ be a
  nontrivial finite family of objects on $G$ with Fulkerson blocker $\hat{\Ga}$. Let the exponent $1<p<\infty$ be given,
  with $q:=p/(p-1)$ its H\"{o}lder conjugate exponent. For any set of weights
  $\sigma\in\mathbb{R}_{>0}^E$ define the dual set of weights
  $\hat{\sigma}$ as $\hat{\sigma}(e) := \sigma(e)^{-\frac{q}{p}}$, for all $e\in E$.
 
  Then
  \begin{equation}
    \label{eq:duality}
    \Mod_{p,\sigma}(\Ga)^{\frac{1}{p}}\Mod_{q,\hat{\sigma}}(\hat{\Ga})^{\frac{1}{q}} = 1.
  \end{equation}
  Moreover, the optimal $\rho^*\in\Adm(\Ga)$ and $\eta^*\in\Adm(\hat{\Ga})$ are unique and are related as follows:
  \begin{equation}\label{eq:rhostaretastar}
   \eta^*(e) =  \frac{\sigma(e)\rho^*(e)^{p-1}}{\Mod_{p,\si}(\Ga)}\qquad\forall e\in E.
  \end{equation}
\end{theorem}
\begin{remark}
  The case for $p=2$, namely 
   \begin{equation*}
    \Mod_{2,\si}(\Ga)\Mod_{2,\sigma^{-1}}(\hat{\Ga})= 1,
  \end{equation*}
  is essentially contained in
  \cite[Lemma 2]{lovasz:jcta2001}, although stated with different terminology and with a different proof.
  In this case~\eqref{eq:rhostaretastar} can be rewritten as
  \begin{equation*}
   \si(e)\rho^*(e) =  \Mod_{2,\si}(\Ga)\eta^*(e)\qquad\forall e\in E.
  \end{equation*}
\end{remark}
\begin{proof}
  For all $\rho\in\Adm(\Ga)$ and $\eta\in\Adm(\hat{\Ga})$, H\"{o}lder's
  inequality implies that
  \begin{equation}
    \label{eq:holder}
    \begin{split}
      1 \le \sum_{e\in E}\rho(e)\eta(e) &= \sum_{e\in
        E}\left(\sigma(e)^{1/p}\rho(e)\right)
      \left(\sigma(e)^{-1/p}\eta(e)\right) \\
      &\le \left( \sum_{e\in E}\sigma(e)\rho(e)^p \right)^{1/p} \left(
        \sum_{e\in E}\hat{\sigma}(e)\eta(e)^{q} \right)^{1/q},
    \end{split}
  \end{equation}
  so
  \begin{equation}
    \label{eq:inequality}
    \Mod_{p,\sigma}(\Ga)^{1/p}\Mod_{q,\hat{\sigma}}(\hat{\Ga})^{1/q} \ge 1.
  \end{equation}

  Now, let $\alpha := \Mod_{q,\hat{\sigma}}(\hat{\Ga})^{-1}$ and let
  $\eta^*\in\Adm(\hat{\Ga})$ be the minimizer for $\Mod_{q,\hat{\sigma}}(\hat{\Ga})$.
  Then~\eqref{eq:inequality} implies that
  \begin{equation}
    \label{eq:lower-bound}
    \Mod_{p,\sigma}(\Ga) \ge \alpha^{\frac{p}{q}} = \alpha^{\frac{1}{q-1}}.
  \end{equation}
  Define
  \begin{equation}\label{eq:rhostaretastar2}
    \rho^*(e) := \alpha\left(
      \frac{\hat{\sigma}(e)}{\sigma(e)}\eta^*(e)^{q}
    \right)^{1/p}
    = \alpha\hat{\sigma}(e)\eta^*(e)^{q/p}.
  \end{equation}
  Note that
  \begin{equation*}
    \cE_{p,\sigma}(\rho^*) = \sum_{e\in E}\sigma(e)\rho^*(e)^p
    = \alpha^p\sum_{e\in E}\hat{\sigma}(e)\eta^*(e)^{q} = 
    \alpha^{p-1} = \alpha^{\frac{1}{q-1}}.
  \end{equation*}
  Thus, if we can show that $\rho^*\in\Adm(\Ga)$,
  then~\eqref{eq:lower-bound} is attained and $\rho^*$ must be
  extremal for $\Mod_{p,\sigma}(\Ga)$. In particular, (\ref{eq:duality})  would  follow. Moreover, (\ref{eq:rhostaretastar}) is another way of writing (\ref{eq:rhostaretastar2}).

  To see that $\rho^*\in\Adm(\Ga)$, we will verify that
  $\sum_{e\in E}\rho^*(e)\eta(e)\ge 1$ for all $\eta\in\Adm(\hat{\Ga})$.
  First, consider $\eta=\eta^*$.  In this case
  \begin{equation*}
    \sum_{e\in E}\rho^*(e)\eta^*(e) = \alpha\sum_{e\in E}\hat{\sigma}(e)\eta^*(e)^q = 1.
  \end{equation*}
  Now let $\eta\in\Adm(\hat{\Ga})$ be arbitrary.  Since $\Adm(\hat{\Ga})$ is
  convex, we have that $(1-\theta)\eta^* + \theta\eta\in\Adm(\hat{\Ga})$
  for all $\theta\in[0,1]$.  So, using Taylor's theorem, we have
  \begin{equation*}
    \begin{split}
      \alpha^{-1} &= \cE_{q,\hat{\sigma}}(\eta^*) \le
      \cE_{q,\hat{\sigma}}((1-\theta)\eta^* + \theta\eta) = \sum_{e\in
        E}\hat{\sigma}(e)\left[(1-\theta)\eta^*(e) +
        \theta\eta(e)\right]^{q} \\
      &= \alpha^{-1} + q\theta\sum_{e\in E}\hat{\sigma}(e)\eta^*(e)^{q-1}
      \left(\eta(e)-\eta^*(e)\right) + O(\theta^2)\\
      &= \alpha^{-1} + \alpha^{-1}q\theta\sum_{e\in E}\rho^*(e)
      \left(\eta(e)-\eta^*(e)\right) + O(\theta^2).
    \end{split}
  \end{equation*}
  Since this inequality must hold for arbitrarily small $\theta>0$, it
  follows that
  \begin{equation*}
    \sum_{e\in E}\rho^*(e)\eta(e) \ge \sum_{e\in E}\rho^*(e)\eta^*(e) = 1,
  \end{equation*}
  and the proof is complete.
\end{proof}

\subsection{The cases $p=1$ and $p=\infty$}

Now we turn our attention to establishing the duality relationship in
the cases $p=1$ and $p=\infty$.  Recall that by Theorem \ref{thm:generalize},
\begin{equation*}
  \lim_{p\to\infty}\Mod_{p,\sigma}(\Ga)^{\frac{1}{p}} = \Mod_{\infty,1}(\Ga) 
  = \frac{1}{\ell(\Ga)},
\end{equation*}
where $\ell(\Ga)$ is defined to be the smallest element of the vector
$\cN\one$.

In order to pass to the limit in~\eqref{eq:duality}, we need to
establish the limits for the second term in the left-hand side
product.
\begin{lemma}
  Under the assumptions of Theorem~\ref{thm:duality},
  \begin{equation}
    \begin{split}
      \lim\limits_{q\to 1}\Mod_{q,
        \hat{\sigma}}(\hat{\Ga})^{\frac{1}{q}} &= \Mod_{1,
        1}(\hat{\Ga})\quad\text{and}\\
      \lim\limits_{q\to \infty}\Mod_{q,
        \hat{\sigma}}(\hat{\Ga})^{\frac{1}{q}} &= \Mod_{\infty,
        \sigma^{-1}}(\hat{\Ga}),
    \end{split}
  \end{equation}
  where $\sigma^{-1}(e)=\sigma(e)^{-1}$.
\end{lemma}
\begin{proof}
  Let $\cN,\hat{\cN}\in\mathbb{R}_{\ge 0}^{\Gamma\times E}$ be the
  usage matrices for  $\Ga$ and $\hat{\Ga}$ respectively.
  Let $\bsi\in\mathbb{R}^{E\times E}$ be the diagonal matrix with
  entries $\bsi(e,e) = \sigma(e)$, and define
  $\tilde{\cN} = \hat{\cN}\bsi$, with $\tilde{\Ga}$ its associated
  family in $\R_{\ge 0}^E$.  Note that $\eta\in\Adm(\hat{\Ga})$ if and only if
  $\bsi^{-1}\eta\in\Adm(\tilde{\Ga})$.  Moreover, for every
  $\eta\in\Adm(\hat{\Ga})$,
  \begin{equation*}
    \cE_{q,\hat{\sigma}}(\eta) = \sum_{e\in E}\hat{\sigma}(e)\eta(e)^q
    = \sum_{e\in E}\sigma(e)\left(\frac{\eta(e)}{\sigma(e)}\right)^q
    = \cE_{q,\sigma}(\bsi^{-1}\eta),
  \end{equation*}
  which implies that
  \begin{equation*}
    \Mod_{q,\hat{\sigma}}(\hat{\Ga}) = \Mod_{q,\sigma}(\tilde{\Ga}).
  \end{equation*}
  Taking the limit as $q\to 1$ and using  the continuity of $p$-modulus with respect to $p$, see  Theorem~\ref{thm:generalize}, we get that
  \begin{equation*}
    \lim_{q\to 1}\Mod_{q,\hat{\sigma}}(\hat{\Ga})^{\frac{1}{q}} 
    =  \lim_{q\to 1}\Mod_{q,\sigma}(\tilde{\Ga})^{\frac{1}{q}} 
    = \Mod_{1,\sigma}(\tilde{\Ga})
    = \min_{\eta\in\Adm(\hat{\Ga})}\sum_{e\in E}\sigma(e)
    \left(\frac{\eta(e)}{\sigma(e)}\right)
    = \Mod_{1,1}(\hat{\Ga}).
  \end{equation*}
  Taking the limit as $q\to\infty$ and using
  Theorem~\ref{thm:generalize} shows that
  \begin{equation*}
    \lim_{q\to\infty}\Mod_{q,\hat{\sigma}}(\hat{\Ga})^{\frac{1}{q}} 
    = \lim_{q\to \infty}\Mod_{q,\sigma}(\tilde{\Ga})^{\frac{1}{q}} 
    = \Mod_{\infty,1}(\tilde{\Ga})
    = \min_{\eta\in\Adm(\hat{\Ga})}\max_{e\in E}\left(\frac{\eta(e)}{\sigma(e)}\right)
    = \Mod_{\infty,\sigma^{-1}}(\hat{\Ga}).
  \end{equation*}
\end{proof}
Taking the limit as $p\to 1$ in Theorem~\ref{thm:duality} then gives
the following theorem.
\begin{theorem}
  \label{thm:dual-infty}
  Under the assumptions of Theorem~\ref{thm:duality},
  \begin{equation}\label{eq:dual-infty}
    \Mod_{1,\sigma}(\Ga)\Mod_{\infty, \sigma^{-1}}(\hat{\Ga}) =1.
  \end{equation}
\end{theorem}
Note that taking the limit as $p\to\infty$ simply yields the same
result for the unweighted case.

\section{Blocking Duality for Families of Objects}\label{sec:examples}

\subsection{Duality for $1$-modulus}
Suppose that $G=(V,E,\si)$ is a weighted graph, with weights $\si\in\R_{> 0}^E$, and $\Ga$ is a non-trivial, finite family of subsets of $E$, where $\cN$ be the corresponding usage matrix. In this case we can equate each $\ga\in\Ga$ with the vector $\ones_{\ga}\in\R_{\ge 0}^E$, so we think of $\Ga$ as living in $\{0,1\}^E\subset\R_{\ge 0}^E$.
Recall that $\Mod_{1,\si}(\Ga)$ is the value of the linear program:
\begin{equation}\label{eq:min-cut}
  \begin{split}
    \text{minimize} &\qquad \si^T\rho \\
    \text{subject to} &\qquad  \rho\ge 0,\quad \cN \rho \geq \one
  \end{split}
\end{equation}
Since this is a feasible linear program, strong duality holds,
and the dual problem is
\begin{equation}\label{eq:max-flow}
  \begin{split}
    \text{maximize} &\qquad \la^T\one \\
    \text{subject to} &\qquad  \la \ge 0,\quad \cN^T \la \leq \si.
  \end{split}
\end{equation}
We think of (\ref{eq:max-flow}) as a (generalized) max-flow problem, given the weights $\si$. 
That's because the condition $\cN^T \la \leq \si$ says that for every $e\in E$
\[
\sum_{\ga\in\Ga}\lambda(\ga)\cN(\ga,e) = \sum_{\stackrel{\ga\in\Ga}{e\in\ga}}\la(\ga)\leq\si(e).
\]
However,
to think of (\ref{eq:min-cut}) as a (generalized) min-cut problem, we would need to be able to restrict the densities $\rho$ to some given subsets of $E$. That's exactly what the Fulkerson  blocker does.
\begin{proposition}\label{prop:mod1mc}
Suppose $G=(V,E)$ is a finite graph and $\Ga$ is a family of subsets of $E$ with Fulkerson blocker family $\hat{\Ga}$.
Then for any set of weights $\si\in\R_{>0}^E$,
\begin{equation}\label{eq:mod1mc}
\Mod_{1,\si}(\Ga)=\min_{\hat{\ga}\in\hat{\Ga}} \sum_{e\in E}\hat{\cN}(\hat{\ga},e)\si(e).
\end{equation}
Moreover, for every $\hat{\ga}\in\hat{\Ga}$ there is a choice of $\si\in \R_{\ge 0}^E$ such that $\hat{\ga}$ is the unique solution of (\ref{eq:mod1mc}).
\end{proposition}
\begin{proof}
By Theorem \ref{thm:fulkerson}(1)
\begin{equation*}
\Adm(\Ga)=\Dom(\hat{\Ga})
\end{equation*}
 So if $\si\in\R_{>0}^E$ is a given set of weights, then, by (\ref{eq:min-cut}), $\Mod_{1,\si}(\Ga)$ is the value of the linear program
\begin{equation}\label{eq:min-cut-dom}
  \begin{split}
    \text{minimize} &\qquad \si^T\rho \\
    \text{subject to} &\qquad  \rho\in \Dom(\hat{\Ga}).
  \end{split}
\end{equation}
In particular, the optimal value is attained at a vertex of $\Dom(\hat{\Ga})$, namely for an object $\hat{\ga}\in\hat{\Ga}$. Therefore, the optimization can be restricted to $\hat{\Ga}$.

The last sentence of the proposition follows from~\cite[Thm.~18.6]{Rockafellar1970} since  $\Adm(\Ga)$ is a recessive polyhedron with finitely many extreme points.
\end{proof}

\begin{remark}
When $\Ga$ is a family of subsets of $E$, it is customary to say that
$\Ga$ {\it has the max-flow-min-cut property}, if its Fulkerson blocker $\hat{\Ga}$ is also a family of subsets of $E$.
For more details we refer to the discussion in \cite[Chapter 3]{lovasz:cbms1986}.  
\end{remark}

\subsection{Connecting families}
\label{sec:connectfam}
Let $G$ be an undirected graph and let $\Ga=\Ga(a,b)$ be the family of all simple paths connecting two distinct nodes $a$ and $b$, i.e., the $ab$-paths in $G$. Consider the family $\Ga_{\rm cut}(a,b)$ of all minimal $ab$-cuts.  Recall that an $ab$-cut $S$ is called minimal if its boundary $\partial S$ does not contain the boundary of any other $ab$-cut as a strict subset.

Note that (\ref{eq:max-flow}) in this case is exactly the max-flow problem.  It is not surprising, then, that~\eqref{eq:min-cut} is closely related to the min-cut problem.  Indeed, every $ab$-cut, $S\subset V$ produces a density $\rho_S := 1_{\partial S}$ that is admissible for~\eqref{eq:min-cut} since every path $\gamma\in\Gamma(a,b)$ must have at least one edge in common with $\partial S$.  Moreover, the max-flow min-cut theorem implies that there exists an $ab$-cut $S$ whose value (i.e., $\sigma^T\rho_S$) equals the value of~\eqref{eq:max-flow}.  Strong duality, then, implies that such a $\rho_S$ minimizes~\eqref{eq:min-cut}.

The last part of Proposition~\ref{prop:mod1mc}, therefore, shows that every element of $\hat{\Gamma}$ is a minimal $ab$-cut.  Conversely, if $\hat{\ga}$ is a minimal $ab$-cut, then we can define $\si$ to be very small on $\hat{\ga}$ and large otherwise, so that $\hat{\ga}$ is the unique solution of the min-cut problem. Therefore, the Fulkerson blocker of $\Ga(a,b)$  is $\hat{\Ga}(a,b)=\Ga_{\rm cut}(a,b)$.

Moreover, the duality
\begin{equation*}
  \Mod_{p,\sigma}(\Gamma)^{\frac{1}{p}}\Mod_{q,\hat{\sigma}}(\hat{\Gamma})^{\frac{1}{q}}=1
\end{equation*}
can be viewed as a generalization of the max-flow min-cut theorem.  
To see this, consider the limiting case~\eqref{eq:dual-infty}.
As discussed above, $\Mod_{1,\si}(\Ga)$
takes the value of the minimum $ab$-cut with edge weights $\sigma$.

With a little work, the second modulus in~\eqref{eq:dual-infty}, can be
recognized as the reciprocal of the corresponding max flow problem.
Using the standard trick for $\infty$-norms, the modulus problem $\Mod_{\infty,\si^{-1}}(\hat{\Ga})$ can
be transformed into a linear program taking the form
\begin{equation*}
  \begin{split}
    \text{minimize} &\qquad t \\
    \text{subject to} &\qquad  \sigma(e)^{-1}\eta(e) \le t\;\forall e\in E\\
    &\qquad \eta\ge 0,\quad \hat{\cN} \eta \geq 1
  \end{split}
\end{equation*}
The minimum must occur somewhere on the boundary of $\Adm(\hat{\Ga})$ and,
therefore, by Theorem~\ref{thm:fulkerson}(2), must take the form
\begin{equation*}
\eta(e)=\sum_{\gamma\in\Gamma}\lambda(\gamma)\ones_{\gamma}(e)\qquad
\lambda(\gamma)\ge 0,\;\sum_{\gamma\in\Gamma}\lambda(\gamma)=1.
\end{equation*}
In other words, the minimum occurs at a unit $st$-flow $\eta$, and the
problem can be restated as
\begin{equation*}
  \begin{split}
    \text{minimize} &\qquad t \\
    \text{subject to} &\qquad  \frac{1}{t}\eta(e) \le \sigma(e)\;
    \forall e\in E\\
    &\qquad \eta\;\text{a unit $st$-flow}
  \end{split}  
\end{equation*}
The minimum is attained when $\frac{1}{t}\eta$ is a maximum $st$-flow
respecting edge capacities $\sigma(e)$; the value of such a flow is
$1/t$, thus establishing the connection between the $\infty$-modulus and the max-flow problem.

\subsection{Spanning tree modulus}

When $\Gamma$ is the set of spanning trees on an unweighted,
undirected graph $G$ with $\cN(\gamma,\cdot)=\ones_\gamma(\cdot)$, the
Fulkerson blocker $\hat{\Ga}$ can be interpreted as the set of (weighted) \emph{feasible
  partitions}~\cite{chopra1989}.
\begin{definition}
  A \textit{feasible partition} $P$ of a graph $G=(V,E)$ is a
  partition of the vertex set $V$ into two or more subsets,
  $\{V_1, \ldots , V_{k_P}\}$, such that each of the induced subgraphs
  $G(V_i)$ is connected. The corresponding edge set, $E_P$, is defined
  to be the set of edges in $G$ that connect vertices belonging to different $V_i$'s.
\end{definition}
The results of~\cite{chopra1989} imply the following theorem.

\begin{theorem}
  Let $G=(V,E)$ be a simple, connected, unweighted, undirected graph
  and let $\Gamma$ be the family of spanning trees on $G$. Then the Fulkerson blocker of $\Ga$ is the set of all vectors
  \begin{equation*}
   \frac{1}{k_P-1}\ones_{E_P}.
  \end{equation*}
ranging over all feasible partitions $P$.
\end{theorem}
This fact plays an important role in \cite{achpcst}.

\section{Blocking Duality and the Probabilistic Interpretation}\label{sec:prob}

At the end of Section~\ref{sec:dual-prob} it was claimed that blocking
duality was closely related to Lagrangian duality.  In this section,
we make this connection explicit.

\begin{theorem}\label{thm:probint}
  Let $G=(V,E,\si)$ be a graph and $\Gamma$ a finite family of objects on
  $G$ with Fulkerson blocker $\hat{\Ga}$.  For a given $1<p<\infty$, let
  $\mu^*$ be an optimal pmf for the minimization problem
  in~\eqref{eq:prob-interp} and let $\eta^*$ be optimal for
  $\Mod_{q,\hat{\sigma}}(\hat{\Ga})$.  Then, in the notation of
  Section~\ref{sec:dual-prob},
  \begin{equation}\label{eq:probabinterp}
    \eta^*(e) = \mathbb{E}_{\mu^*}\left[\cN(\underline{\gamma},e)\right].
  \end{equation}
\end{theorem}

\begin{proof}
  Every $\eta\in\Adm(\hat{\Ga})$ can be written as the sum of a convex
  combination of the vertices of $\Adm(\hat{\Ga})$ and a nonnegative vector.
  In other words, $\eta\in\Adm(\hat{\Ga})$ if and only if there exists
  $\mu\in\cP(\Gamma)$ and $\eta_0\in\mathbb{R}_{\ge 0}^E$ such that
  $\eta = \cN^T\mu + \eta_0$.  Or, in probabilistic notation,
  \begin{equation*}
    \eta(e) = \sum_{\gamma\in\Gamma}\cN(\gamma,e)\mu(\gamma) + \eta_0(e)
    = \mathbb{E}_{\mu}\left[\cN(\underline{\gamma},e)\right] + \eta_0(e).
  \end{equation*}
  For such an $\eta$,
  \begin{equation*}
    \cE_{q,\hat{\sigma}}(\eta) = \sum_{e\in E}\sigma(e)^{-\frac{q}{p}}\eta(e)^q
    \ge \sum_{e\in E}\sigma(e)^{-\frac{q}{p}} 
    \mathbb{E}_{\mu}\left[\cN(\underline{\gamma},e)\right]^q
  \end{equation*}
  with equality holding if and only if $\eta_0=0$.  This implies that
  the optimal $\eta^*$ must be of the form
  $\eta^*=\cN^T\mu'=\mathbb{E}_{\mu'}\left[\cN(\underline{\gamma},\cdot)\right]$
  for some $\mu'\in\cP(\Gamma)$.

  Now, let $\mu^*$ be any optimal pmf for~\eqref{eq:prob-interp} and let $\eta'=\cN^T\mu^*$.  Since $\eta'=\cN\mu^*\in\Dom(\Gamma)$, Theorem~\ref{thm:fulkerson}(2) implies that $\eta'\in\Adm(\hat{\Ga})$.  Moreover, by optimality of $\mu^*$,
  \begin{equation*}
    \cE_{q,\hat{\sigma}}(\eta') =
    \sum_{e\in E}\sigma(e)^{-\frac{q}{p}} 
    \mathbb{E}_{\mu^*}\left[\cN(\underline{\gamma},e)\right]^q
    \le 
    \sum_{e\in E}\sigma(e)^{-\frac{q}{p}} 
    \mathbb{E}_{\mu'}\left[\cN(\underline{\gamma},e)\right]^q
    = \cE_{q,\hat{\sigma}}(\eta^*).
  \end{equation*}
  But, since $1<q<\infty$, the minimizer for
  $\Mod_{q,\hat{\sigma}}(\hat{\Ga})$ is unique and, therefore,
  $\eta'=\eta^*$. So
  $\eta^*=\cN^T\mu^*=\mathbb{E}_{\mu^*}\left[\cN(\underline{\gamma},\cdot)\right]$
  as claimed.
\end{proof}

\section{The $\de_p$ metrics and a new proof that effective resistance is a metric}\label{sec:metrics}

We saw in Theorem \ref{thm:generalize} that in the case of connecting families $\Mod_{p,\si}(\Ga(a,b))$ satisfies: 
\begin{itemize}
\item $\Mod_{\infty,1}(\Ga(a,b))^{-1}=\ell(\Ga(a,b))$ is the (unweighted) shortest-path length;
\item $\Mod_{2,\si}(\Ga(a,b))^{-1}=\cReff(a,b)$ is the effective resistance metric;
\item $\Mod_{1,\si}(\Ga(a,b))^{-1}=\MC(a,b)^{-1}$ is the reciprocal of mincut.
\end{itemize}
In all three cases, if $G$ is a connected graph, these are distances (or metrics). The fact that shortest-path $d_{\rm SP}(a,b):=\ell(\Ga(a,b))$ is a metric on $V$ is well known and follows easily from the definition.
 
The fact that $d_{\rm MC}(a,b):=\MC(a,b)^{-1}$ is an ultrametric (i.e. that the sum can be replaced by the maximum in the triangle inequality) is left as an exercise, or see \cite{afpc} where a proof is given.

The fact that effective resistance $d_{\rm ER}(a,b):=\cReff(a,b)$ is a metric has several known proofs. See \cite[Exercise 9.8]{levin-peres-wilmer2009}, for a proof using current flows, and see \cite[Corollary 10.8]{levin-peres-wilmer2009}, for one using commute times. 
As a consequence of Theorem \ref{thm:deltap}, we will provide yet another proof that effective resistance is a metric on graphs.

\begin{definition}
Let $G=(V,E,\si)$ be a weighted, connected, simple graph. Given $a,b\in V$, let $\Ga(a,b)$ be the connecting family of all paths between $a$ and $b$. Fix $1<p<\infty$ and let $q:=p/(p-1)$ be the H\"{o}lder conjugate exponent. Then we define
\[
\de_p(a,b):=
\begin{cases}
  0 & \text{if }a = b,\\
  \Mod_{p,\si}(\Ga(a,b))^{-q/p} &\text{if }a\ne b.
\end{cases}
\]
\end{definition}

\begin{theorem}\label{thm:deltap}
Suppose $G=(V,E,\si)$ is a weighted, connected, simple graph. Then $\de_p$ is a metric on $V$. Moreover, 
\begin{itemize}
\item[(a)] $\lim_{p\uparrow\infty}\de_p= d_{\rm SP}$;
\item[(b)] $\de_2=d_{\rm ER}$; 
\item[(c)] For $1<p<2$, $\Mod_{p,\si}(\Ga(a,b))^{-1}$ is a metric and it tends to $d_{\MC}(a,b)$ as $p\rightarrow 1$.
\end{itemize}
Finally, for every $\ep>0$ and every $p\in [1,\infty]$ there is a connected graph for which $\de_p^{1+\ep}$ is not a metric.
\end{theorem}
\begin{remark}
Note that, in light of Theorem~\ref{thm:generalize}, when $p=2$, the proof of Theorem \ref{thm:deltap} gives an alternative modulus-based proof that effective resistance is a metric.
\end{remark}
\begin{remark}
It is straightforward to show that an arbitrary positive power of an ultrametric is also an ultrametric, so $(d_{\rm MC})^t$ is a metric for any $t>0$.  Using~\eqref{eq:monotone-decr} and~\eqref{eq:monotone-incr} it can be shown that as $p\downarrow 1$, $\de_p$  converges to the limit
\begin{equation*}
  \lim_{t\to\infty}(d_{\rm MC}(a,b))^t =
  \begin{cases}
    0 & \text{if }d_{\rm MC}(a,b) > 1,\\
    1 & \text{if }d_{\rm MC}(a,b) = 1,\\
    \infty & \text{if }d_{\rm MC}(a,b) < 1.
 \end{cases}
\end{equation*}
For unweighted graphs, this limit essentially decomposes the graph into its 2-edge-connected components.  All nodes in the same component are distance zero from one another while nodes in different components are at distance one.
\end{remark}
\begin{proof} 
Assuming the claim that $\de_p$ is a metric, the `Moreover' parts (a) and (b) follow from Theorem \ref{thm:generalize}. For (c), recall that a metric $d$ can always be raised to an exponent $0<\ep<1$ and still remain a metric. Since for $1<p<2$, we have $p/q<1$, it follows that
$\Mod_{p,\si}(\Ga(a,b))^{-1}=\de_p^{p/q}$ is a metric, and the claim follows from continuity in $p$.
Finally, the fact that the exponent $1$ is sharp for the metrics $\de_p$ is shown in \cite{afpc}. For completeness, we repeat the argument here. 
  Consider the (unweighted) path graph $P_3$ with edges  $\{a,c\},\{c,b\}$ and fix
  $p\in (1,\infty)$.  First $\Mod_p(\Ga(a,c)) = 1$, because any
  admissible density $\rho$ must satisfy $\rho(a,c)=1$, furthermore, to minimize the energy, we also set $\rho(c,b)=0$. Likewise, $\Mod_p(\Ga(c,b)) =1$.  For $\Mod_p(\Ga(a,b))$, the energy
  is minimized when $\rho(a,c)=\rho(c,b)=1/2$. Thus,
\[
\Mod_p(\Ga(a,b)) = (1/2)^p + (1/2)^p = 2^{1-p}
\]
Hence, $ \de_p(a,b) = 2^{q(p-1)/p} = 2=1+1=\de_p(a,c)+\de_p(c,b)$.  In particular, the triangle inequality will fail for $\de_p^t$ as soon as $t>1$.

The proof of the main claim hinges on the dual formulation in terms of Fulkerson blocker duality. Fix $p\in (1,\infty)$. Recall from Section \ref{sec:connectfam}, that the Fulkerson blocker family for $\Ga(a,b)$ is the family of all minimal $ab$-cuts $\hat{\Ga}(a,b)$.
An important observation at this point is that the word `minimal' can be omitted without changing the modulus problem: since every $ab$-cut contains a minimal $ab$-cut, any $\eta$ that is admissible for the minimal cut family is admissible for the family of all $ab$-cuts.  Without loss of generality, then, we consider $\hat{\Ga}(a,b)$ to be the set of all $ab$-cuts.  By Theorem \ref{thm:duality},
\[
\Mod_{p,\si}(\Ga(a,b))^{-q/p}=\Mod_{q,\hat{\si}}(\hat{\Ga}(a,b)),
\]
where $q:=p/(p-1)$ is the H\"older conjugate exponent of $p$ and $\hat{\si}=\si^{-q/p}$.

Now suppose $a,b,c\in V$ are distinct. Then, for every $ab$-cut $S\in\hat{\Ga}(a,b)$, we have the following mutually exclusive cases: either $c\in S$ or $c\not \in S$. Therefore, 
\[
\hat{\Ga}(a,b)\subset\hat{\Ga}(a,c)\cup\hat{\Ga}(c,b).
\]
The triangle inequality then follows from monotonicity (\ref{eq:monotonicity}) and  subadditivity (\ref{eq:subadditivity}) of modulus:
\begin{align*}
\de_p(a,b) & = \Mod_{p,\si}(\Ga(a,b))^{-q/p}&\text{(Definition)}\\
& =\Mod_{q,\hat{\si}}(\hat{\Ga}(a,b))& \text{(Fulkerson duality)}\\
& \leq\Mod_{q,\hat{\si}}(\hat{\Ga}(a,c)\cup\hat{\Ga}(c,b))& \text{(Monotonicity)}\\
& \leq \Mod_{q,\hat{\si}}(\hat{\Ga}(a,c))+\Mod_{q,\hat{\si}}(\hat{\Ga}(c,b)) &\text{(Subadditivity)}\\
& =\de_p(a,c)+\de_p(c,b).&\text{(Fulkerson duality)}
\end{align*}
Verifying the remaining metric axioms  is left to the reader.
\end{proof}

\section{Edge-conductance monotonicity}\label{sec:monotonicity}

When studying the $p$-modulus of a family of objects $\Ga$ on a weighted graph $G=(V,E,\si)$, we often refer to the weights $\si(e)$ as edge-conductances. This terminology originates in the special case of connecting families $\Ga(a,b)$ on undirected graphs with $p=2$. In that case, $\Mod_{2,\si}(\Ga(a,b))$ coincides with effective conductance and we can give an electrical network interpretation to the various quantities of interest.
In particular, the optimal density $\rho^*(e)$ represents the absolute voltage potential drop across $e$, $\si(e)$ is the conductance of $e$, and therefore $\si(e)\rho^*(e)$ is the current flow across $e$ (by Ohm's law). Moreover, recall
the optimal density for the Fulkerson blocker $\eta^*(e)$, which probabilistically is the expected usage of $e$ by random paths under an optimal pmf (see Theorem \ref{thm:probint}). We know that $\eta^*(e)$ is related to $\rho^*(e)$ via (\ref{eq:rhostaretastar}), which can be written in this case as
\[
\eta^*(e)=\frac{\si(e)\rho^*(e)}{\Mod_{2,\si}(\Ga(a,b))}.
\]
Therefore, $\eta^*(e)$ is proportional to the current flow across $e$. And
\[
\rho^*(e)\eta^*(e)=\frac{\si(e)\rho^*(e)^2}{\sum_{e'\in E}\si(e')\rho^*(e')^2}
\]
is the fraction of the total dissipated power due to the resistor on edge  $e$.

In the theory of electrical networks, the following edge-conductance monotonicity property is well known, see for instance Spielman's notes \cite[Problem 4]{spielmanpb4}.
\begin{proposition}\label{prop:spielman}
Let $G = (V, E)$  be an undirected, connected graph and let $r$ be the edge resistances. Let $e$
be an edge of $E$ and let $\tilde{r}$ be another set of resistances such that $\tilde{r}(e') = r(e')$, for all  $e'\neq e$, and
$\tilde{r}(e) \geq r(e)$.
Fix an edge  $\{s, t\}$ of $G$. If one unit of current flows from $s$ to $t$, the amount of
current that flows through edge $e$ under resistances $\tilde{r}$ is no larger than the amount that flows under resistances
$r$.
\end{proposition} 
Our goal is to generalize Proposition \ref{prop:spielman} to $p$-modulus of arbitrary families of objects.
In the language of modulus, Proposition \ref{prop:spielman} says that if $\{s,t\}$ is an edge in $E$ and we are trying to computing $\Mod_{2,\si}(\Ga(s,t))$, then lowering $\si(e)$ on some edge $e\in E$ results in a new modulus problem $\Mod_{2,\tilde{\si}}(\Gamma(s,t))$ whose extremal density satisfies $\rho^*_{\tilde{\si}}(e)\le \rho^*_\si(e)$.

Theorem \ref{thm:concave} below is a reformulation, in the context of general families of objects, of results from \cite[Section 6.2]{abppcw:ecgd2015} that  were formulated in terms of families of walks. In order, to keep the flow of the paper intact, we have relegated the proof of Theorem \ref{thm:concave} to the Appendix.
\begin{theorem}[\cite{abppcw:ecgd2015}]\label{thm:concave}
Let $G=(V,E,\si)$ be a graph and $\Ga$ a non-empty and non-trivial finite family of objects on $G$. Fix $1< p<\infty$ and let $\rho^*_\si$ be the extremal density for $\Mod_{p,\si}(\Ga)$. Then 
\begin{enumerate}
\item the map
$\phi: \R_{> 0}^E\rightarrow \R$ given by $\phi(\si):=\Mod_{p,\si}(\Ga)$ is Lipschitz continuous; 
\item the extremal density $\rho^*_\si$ is also continuous in $\si$;
\item the map $\phi$ is concave;
\item the map $\phi$ is differentiable, and the partial derivatives of $\phi$ satisfy 
\[
\frac{\bd \phi}{\bd \si(e)}=\rho^*_{\si}(e)^p\qquad\forall e\in E.
\]
\end{enumerate}
\end{theorem}
\begin{theorem}\label{thm:monotonicity}
Under the hypothesis of Theorem \ref{thm:concave}, with $\eta^*_\si$ given by (\ref{eq:rhostaretastar}), we have that in each variable $\si(e)$,
\begin{itemize}
\item[(a)] $\Mod_{p,\si}(\Ga)$ is weakly increasing.
\item[(b)] $\rho^*_\si(e)$ is weakly decreasing.
\item[(c)] $\eta^*_\si(e)$ is weakly increasing.
\end{itemize}
\end{theorem}
\begin{remark}
Note that Theorem \ref{thm:monotonicity} (c), can be reformulated using the probabilistic interpretation (\ref{eq:probabinterp}) as saying that if $\si(e)$ increases (and the other weights are left alone), then the expected usage of edge $e$ increases.
\end{remark}
\begin{proof}[Proof of Theorem \ref{thm:monotonicity}]
For part (a),  by Theorem \ref{thm:concave} (1), $\Mod_{p,\si}(\Ga)$ is absolutely continuous in $\si(e)$. In particular, the fundamental theorem of calculus holds and the result follows from  Theorem \ref{thm:concave} (4).

For part (b), write $f(h):=\Mod_{p,\si_h}(\Ga)$, where $\si_h:=\si+h\ones_{e}$. Set $h>0$. Then, by concavity and differentiability (Theorem \ref{thm:concave} (3) and (4)),
\[
f'(0)\ge \frac{f(h)-f(0)}{h}\ge f'(h). 
\]
The result follows from  Theorem \ref{thm:concave} (4) since
\begin{equation*}
  f'(h) = \frac{\partial}{\partial \sigma_h(e)}\phi(\sigma_h) = \rho_{\sigma_h}^*(e)^p.
\end{equation*}

Note that (\ref{eq:rhostaretastar}) is not sufficient to prove part (c), since
it's not immediately clear how the right-hand side varies with $\si(e)$. Instead, we use the fact that, by Theorem~\ref{thm:duality}, $\eta^*_\si$ is the optimal density for $\Mod_{q,\hat{\si}}(\hat{\Ga})$ where $\hat{\si}=\si^{-q/p}$ (a smooth decreasing function of $\si$), and use part (b).
\end{proof}

\section{Randomly weighted graphs}\label{sec:lovasz}
In this section we explore the main arguments in \cite{lovasz:jcta2001} and recast them in the language of modulus.
The goal is to study graphs $G=(V,E,\si)$ where the weights $\si\in\R_{>0}^E$ are random variables and compare modulus computed on $G$ to the corresponding modulus computed on the deterministic graph $\bE G:=(V,E,\bE\si)$. Theorem \ref{thm:lovasz} below is a reformulation of Theorem 7 in \cite{lovasz:jcta2001}, which generalized Theorem 2.1 in
\cite{lyons-pemantle-peres:jcta1999}.  In Theorem~\ref{thm:bounds}, we combine Theorem~\ref{thm:lovasz} with the monotonicity properties in Theorem~\ref{thm:generalize} to obtain a new lower bound for the expected $p$-modulus in terms of $p$-modulus on $\bE G$.

First we recall a lemma from Lov{\'a}sz's paper.
\begin{lemma}[{\cite[Lemma 9]{lovasz:jcta2001}}]\label{lem:expmin}
Let $W\in\R_{> 0}^E$ be a random variable with survival function
\[S(t):=\bP\left( W\ge t\right), \qquad \text{for }t\in\R_{\ge 0}^E.\]
If $S(t)$ is log-concave, then the survival function of $\min_{e\in E} W(e)$ is also log-concave and 
$W$ satisfies 
\begin{equation}\label{eq:expmin}
\bE\left(\min_{e\in E} W(e)\right)\ge \left(\sum_{e\in E} \frac{1}{\bE(W(e))}\right)^{-1}.
\end{equation}
\end{lemma}
Property (\ref{eq:expmin}) is satisfied if for instance the random variables $\{W(e)\}_{e\in E}$ are independent and distributed as exponential variables ${\rm Exp}(\la(e))$, i.e., so that $\bP(W(e)>t)=\min\{\exp(-\la(e)t), 1\}$.

It is useful to collect some properties of random variables with log-concave survival functions.
\begin{proposition}\label{prop:logconcavity}
Let $W\in\R_{> 0}^E$ be a random variable with log-concave survival function. Then the following random variables also have log-concave survival function:
\bi
\item[(a)] $CW$, where $C=\Diag(c(\cdot))$ with $c\in\R_{>0}^E$.
\item[(b)]  $W^*$, where  $E^*\subset E$, and $W^*\in\R_{>0}^{E^*}$ is the projection of $W$ onto $\R_{>0}^{E^*}$.
\ei
\end{proposition}
\begin{proof}
We define $S(t):=\bP(W\ge t)$ for $t\in \R_{\ge 0}^E$.
For (a), note that
\[
\log \bP\left( CW \ge t\right)= \log S(C^{-1}t),
\]
which is the composition of a concave function with an affine function.
Likewise (b) follows by composing a concave function with a projection.
\end{proof}

\begin{theorem}\label{thm:lovasz}
Let $G=(V,E,\si)$ be a simple finite graph. Assume the $\si$ is a random variable in $\R_{>0}^E$ with the property that
its survival function is log-concave.
Let $\Ga$ be a finite non-trivial family of objects on $G$, with $\cN_{\rm min}$ defined as in (\ref{eq:usagelb}).
Then
\[
\bE\Mod_{1,\si}(\Ga)\ge \cN_{\rm min}\Mod_{2,\bE\si}(\Ga).
\]
\end{theorem}

\begin{proof}
Let $\hat{\Ga}$ be the Fulkerson blocker of $\Ga$. Let $\rho^*$ be extremal for $\Mod_{2,\bE\si}(\Ga)$ and $\eta^*$ be  extremal for $\Mod_{2,(\bE\si)^{-1}}(\hat{\Ga})$. Also let $\mu^*\in\cP(\Ga)$ be an optimal measure, then we know that
\begin{equation}\label{eq:def-etastar}
   \eta^*(e) =  \frac{\bE\si(e)\rho^*(e)}{\Mod_{2,\bE(\si)}(\Ga)}=\sum_{\ga\in\Ga}\mu^*(\ga)\cN(\ga,e)=\bE_{\mu^*}\left(\cN(\underline{\ga},e)\right),\qquad\forall e\in E.
  \end{equation}
To avoid dividing by zero let $E^*:=\{e\in E: \eta^*(e)>0\}$ and let $\Ga^*:=\{\ga\in\Ga: \mu^*(\ga)>0\}$. Note that, if $e\not\in E^*$, then
\[
0=\eta^*(e)=\sum_{\ga\in\Ga^*}\mu^*(\ga)\cN(\ga,e),
\]
hence $\cN(\ga,e)=0$ for all $\ga\in\Ga^*$. Therefore, for any $\rho\in\Adm(\Ga)$ and $\ga\in\Ga^*$,
\begin{equation}\label{eq:ellrhoestar}
\sum_{e\in E^*}\cN(\ga,e)\rho(e)=\sum_{e\in E}\cN(\ga,e)\rho(e)=\ell_\rho(\ga) \ge 1.
\end{equation}
Now, fix an arbitrary $\rho\in\Adm(\Ga)$. Then, by  (\ref{eq:def-etastar}), 
\begin{equation}\label{eq:mainstep}
\cE_{1,\si}(\rho) \ge \sum_{e\in E^*}\si(e)\rho(e) =  \Mod_{2,\bE\si}(\Ga)\sum_{e\in E^*}\si(e)\rho(e)\frac{1}{\bE\si(e)\rho^*(e)}\bE_{\mu^*}\left(\cN(\underline{\ga},e)\right),
\end{equation}
where the denominator is positive since $\si>0$ and since $\rho^*> 0$ on $E^*$ by~\eqref{eq:def-etastar}.  Note that
\begin{align*}
\sum_{e\in E^*}\si(e)\rho(e)\frac{1}{\bE\si(e)\rho^*(e)}\bE_{\mu^*}\left(\cN(\underline{\ga},e)\right)   & =  \sum_{\ga\in\Ga^*}\mu^*(\ga)\sum_{e\in E^*}\frac{\si(e)}{\bE\si(e)\rho^*(e)}\cN(\ga,e) \rho(e) \\
& \ge \sum_{\ga\in\Ga^*}\mu^*(\ga)\min_{\substack{e\in E^*\\ \cN(\ga,e)\neq 0}}\frac{\si(e)}{\bE\si(e)\rho^*(e)}\sum_{e\in E^*}\cN(\ga,e)\rho(e) \\
& \ge \sum_{\ga\in\Ga^*}\mu^*(\ga)\min_{\substack{e\in E^*\\ \cN(\ga,e)\neq 0}}\frac{\si(e)}{\bE\si(e)\rho^*(e)}, 
\end{align*}
where the last inequality follows by (\ref{eq:ellrhoestar}).

Minimizing in (\ref{eq:mainstep}) over $\rho\in\Adm(\Ga)$ we find
\begin{equation}\label{eq:pre-expectation}
\Mod_{1,\si}(\Ga)  \ge  \Mod_{2,\bE\si}(\Ga)\sum_{\ga\in\Ga}\mu^*(\ga)\min_{\substack{e\in E^*\\ \cN(\ga,e)\neq 0}}\frac{\si(e)}{\bE\si(e)\rho^*(e)}
\end{equation}
Note that for each $\ga\in\Ga^*$,
by Proposition (\ref{prop:logconcavity}) (a) and (b) and Lemma \ref{lem:expmin},  the scaled random variables 
\[
X(e):=\frac{\si(e)}{\bE\si(e)\rho^*(e)}\qquad \text{for }e\in E^*\text{ with } \cN(\ga,e)\neq 0,
\]
have the property that
\[
\bE\left(\min_{\substack{e\in E^*\\ \cN(\ga,e)\neq 0}}X(e)\right)\ge  \left(\sum_{\substack{e\in E^*\\ \cN(\ga,e)\neq 0}} \frac{1}{\bE(X(e))}\right)^{-1}= \left(\sum_{\substack{e\in E^*\\ \cN(\ga,e)\neq 0}} \rho^*(e)\right)^{-1} .
\]
Moreover, by (\ref{eq:usagelb}), 
\[
 \left(\sum_{\substack{e\in E^*\\ \cN(\ga,e)\neq 0}} \rho^*(e)\right)^{-1}\ge \cN_{\rm min}\left(\sum_{\substack{e\in E^*\\ \cN(\ga,e)\neq 0}}\cN(\ga,e)\rho^*(e)\right)^{-1}.
\]
Finally, by complementary slackness, since $\ga\in\Ga^*$, we have $\mu^*(\ga)>0$, hence
\[
\sum_{\substack{e\in E^*\\ \cN(\ga,e)\neq 0}} \cN(\ga,e)\rho^*(e) = \sum_{e\in E} \cN(\ga,e)\rho^*(e) = 1.
\]
Taking the expectation on both sides of~\eqref{eq:pre-expectation} gives the claim.
\end{proof}

Theorem \ref{thm:lovasz} has some interesting consequences for $p$-modulus on randomly weighted graphs.
First, recall from Theorem \ref{thm:concave} (3) that the map 
\[
\si\mapsto \Mod_{p,\si}(\Ga)
\]
is concave for $1\le p<\infty$. In particular, if $\si\in\R_{>0}^E$ is a random variable, then by Jensen's inequality:
\begin{equation}\label{eq:jensen}
\bE\Mod_{p,\si}(\Ga)\le\Mod_{p,\bE\si}(\Ga).
\end{equation}
The following theorem gives a lower bound.
\begin{theorem}\label{thm:bounds}
Let $G=(V,E,\si)$ be a simple finite graph. Assume $\si$ is a random variable in $\R_{>0}^E$ with log-concave survival function.
Let $\Ga$ be a finite non-trivial family of objects on $G$ with $\cN_{\rm min}$ defined as in (\ref{eq:usagelb}).
Then, for $1\le p \le 2$,
\begin{equation}\label{eq:2mod-lovasz}
\bE\Mod_{p,\si}(\Ga)\ge \frac{\cN_{\rm min}^p}{\bE\si(E)}\Mod_{p,\bE\si}(\Ga)^2.
\end{equation}
\end{theorem}
\begin{proof}
When $1<p\le 2$ we have, by~\eqref{eq:monotone-incr},
\[
\Mod_{2,\bE\si}(\Ga)\ge \bE\si(E)^{1-2/p}\Mod_{p,\bE\si}(\Ga)^{2/p}.
\]
So by Theorem \ref{thm:lovasz} we get
\begin{equation}
\label{eq:lovasz-1p}
\bE\Mod_{1,\si}(\Ga)  \ge  \cN_{\rm min}\bE\si(E)^{1-2/p}\Mod_{p,\bE\si}(\Ga)^{2/p}.
\end{equation}
Letting $p\rightarrow 1$ and by continuity in $p$ (Theorem~\ref{thm:generalize}) we get
\[
\bE\Mod_{1,\si}(\Ga)  \ge  \frac{\cN_{\rm min}}{\bE\si(E)}\Mod_{1,\bE\si}(\Ga)^2
\]
Moreover, estimating the $1$-modulus in terms of $p$-modulus, using~\eqref{eq:monotone-incr} a second time, and then applying H\"{o}lder's inequality gives
\begin{equation}\label{eq:holder-1p}
\bE\Mod_{1,\si}(\Ga)
\le \bE\left( \si(E)^{1/q}\Mod_{p,\si}(\Ga)^{1/p} \right)
\le \bE\si(E)^{1/q}\bE\left(\Mod_{p,\si}(\Ga)\right)^{1/p}.
\end{equation}
Combining (\ref{eq:lovasz-1p}) and (\ref{eq:holder-1p}) gives~\eqref{eq:2mod-lovasz}.
\end{proof}
\begin{remark}
By combining (\ref{eq:jensen}) with (\ref{eq:2mod-lovasz}) we find that, for $1\le p\le 2$,
\[
\Mod_{p,\bE\si}(\Ga)\leq \frac{\bE\si(E)}{\cN_{\rm min}^p}.
\]
This is not a contradiction because this inequality is always satisfied, since the constant density $\rho\equiv\cN_{\rm min}^{-1}$ is always admissible.
\end{remark}
Theorem \ref{thm:bounds} leads one to wonder what lower bounds can be established for $\bE\Mod_{2,\si}(\Ga)$ when $\si$ is allowed to vanish and its survival function is not necessarily log-concave. For instance, it would be interesting to study what happens when the weights $\si(e)$ are independent Bernoulli variables, namely when $G$ is an Erd\H{o}s-R\'enyi graph. The situation there is complicated by the fact that the family $\Ga$ will change with every new sample of the weights $\si$. For instance, the family of all spanning trees will be different for different choices of $\si$.

\section{Appendix}

Here we give a proof of Theorem \ref{thm:concave}, which is a generalization of results in \cite[Section 6.2]{abppcw:ecgd2015}. First recall the following weaker version of Clarkson's inequalities.
\begin{proposition}\label{prop:clarkineq}
Let $1<p<\infty$. Set $M:=\max\{p,q\}$, where $p+q=pq$. Then, for any $f,g\in L^p$,
\begin{equation}\label{eq:wkclarkineq}
\left\|\frac{f+g}{2}\right\|_p^M+\left\|\frac{f-g}{2}\right\|_p^M\leq \frac{ \|f\|_p^M + \|g\|_p^M }{2}.
\end{equation}
\end{proposition}
Next we translate Proposition \ref{prop:clarkineq} in the language of $p$-modulus.
\begin{lemma}\label{lem:quasimin}
Let $G=(V,E,\si)$ be a graph and $\Ga$ a finite, non-empty and non-trivial family of objects on $G$. Let $1<p<\infty$ and set $M:=\max\{p,q\}$, where $pq=p+q$. Let $\rho^*$ denote the (unique) extremal density for $\Mod_{p,\si}(\Ga)$. Then, for every $\rho\in\Adm(\Ga)$:
\[
\|\rho-\rho^*\|_p^M\leq 2^{M-1}\si_{\rm min}^{-M/p}\left(\cE_{p,\si}(\rho)^{M/p}-\Mod_{p,\si}(\Ga)^{M/p}  \right)
\]
where $\si_{\rm min}=\min_{e\in E}\si(e)$.
\end{lemma}
In particular, if $\rho$ is almost a minimizer, then $\rho$ must be close to $\rho^*$.
\begin{proof}
Let $f(e):=\si(e)^{1/p}\rho(e)$ and $f^*(e)=\si(e)^{1/p}\rho^*(e)$. Then
\[
\|f\|_p^p=\cE_{p,\si}(\rho)\qquad\text{and}\qquad \|f^*\|_p^p=\Mod_{p,\si}(\Ga).
\]
Also
\[
\|f-f^*\|_p^p =\sum_{e\in E} \si(e)|\rho(e)-\rho^*(e)|^p\ge \si_{\rm min}\|\rho-\rho^*\|_p^p,
\]
and, since $\Adm(\Ga)$ is convex,
\[
\left\|\frac{f+f^*}{2}\right\|_p^p=\sum_{e\in E}\si(e) \left|\frac{\rho(e)+\rho^*(e)}{2}\right|^p=\cE_{p,\si}\left(\frac{\rho+\rho^*}{2}\right)\geq \Mod_{p,\si}(\Ga).
\]
Applying (\ref{eq:wkclarkineq}) to $f$ and $f^*$ and substituting, we obtain
\[
\Mod_{p,\si}(\Ga)^{M/p}+\si_{\rm min}^{M/p}2^{-M}\|\rho-\rho^*\|_p^{M}\leq\frac{1}{2}\left(\cE_{p,\si}(\rho)^{M/p} +\Mod_{p,\si}(\Ga)^{M/p} \right)
\]
\end{proof}
We are now ready to prove Theorem \ref{thm:concave}.
\begin{proof}[Proof of Theorem \ref{thm:concave}]
To show (1),
fix $\si_1,\si_2\in\R_{>0}^E$. Assume first that $\Mod_{p,\si_2}(\Ga)\le \Mod_{p,\si_1}(\Ga)$. Also recall that $\rho^*_{\si_2}\le\cN_{\rm min}^{-1}$, by Remark \ref{rem:remarks}(b). We have
\begin{align*}
|\Mod_{p,\si_1}(\Ga)-\Mod_{p,\si_2}(\Ga)| &
\leq \cE_{p,\si_1}(\rho^*_{\si_2}) - \Mod_{p,\si_2}(\Ga)\\
&=\cE_{p,\si_1}(\rho^*_{\si_2}) - \cE_{p,\si_2}(\rho^*_{\si_2}) \\
& = \sum_{e\in E} (\si_1(e)-\si_2(e))\rho^*_{\si_2}(e)^p\\
& \le \cN_{\rm min}^{-p}\|\si_1-\si_2\|_1
\end{align*}
A similar argument holds when $\Mod_{p,\si_1}(\Ga)\le \Mod_{p,\si_2}(\Ga)$.
This establishes the Lipschitz continuity of the map $\phi$.

To show (2), by Lemma \ref{lem:quasimin} and $M\geq 2$,
\[
\|\rho^*_{\si_2}-\rho^*_{\si_1}\|_p^M\leq 2^{M-1}\si_{1, {\rm min}}^{-M/p}\left(\cE_{p,\si_1}(\rho^*_{\si_2})^{M/p}-\Mod_{p,\si_1}(\Ga)^{M/p}  \right).
\]
So it's enough to show that 
\[
\cE_{p,\si_1}(\rho^*_{\si_2})\longrightarrow
\Mod_{p,\si_1}(\Ga)\qquad\text{as $\si_2\rightarrow\si_1$}
\]
But
\begin{align*}
\Mod_{p,\si_1}(\Ga) &\leq \cE_{p,\si_1}(\rho^*_{\si_2})=\sum_{e\in E}\si_1(e)\rho^*_{\si_2}(e)^p\\
&=\sum_{e\in E}\si_2(e)\rho^*_{\si_2}(e)^p + \sum_{e\in E}(\si_1(e)-\si_2(e))\rho^*_{\si_2}(e)^p\\
&\le \Mod_{p,\si_2}(\Ga) +
 \cN_{\rm min}^{-p} \|\si_1-\si_2\|_1.
\end{align*}
And, as $\si_2\rightarrow\si_1$, the last line converges to $\Mod_{p,\si_1}(\Ga)$ by continuity of the map $\phi$ in part (1).

To show (3), fix $\si_0,\si_1\in\R_{>0}^E$ and $t\in [0,1]$. Let $\rho^*_t$ be extremal for $\si_t:=t\si_1+(1-t)\si_0$.
Then, since $\rho^*_t\in\Adm(\Ga)$,
\begin{align*}
t\Mod_{p,\si_1}(\Ga)+(1-t)\Mod_{p,\si_0}(\Ga) &\le t\cE_{p,\si_1}(\rho^*_t)+(1-t)\cE_{p,\si_0}(\rho^*_t)\\
& =\sum_{e\in E} \si_t(e) \rho^*_t(e)^p = \Mod_{p,\si_t}(\Ga).
\end{align*}
This proves concavity.

To show (4), fix $\tau\in\R^E$ and let $\ep>0$. Set $\si_\ep:=\si+\ep\tau$. Note that for $\ep$ small enough $\si_\ep\in\R_{>0}^E$. Let $\rho^*_\ep$ be the extremal density corresponding to $\si_\ep$. For any $\rho\in\R_{\ge 0}^E$:
\[
\cE_{p,\si_\ep}(\rho)=\sum_{e\in E} (\si(e)+\ep\tau(e))\rho(e)^p=\cE_{p,\si}(\rho)+\ep\cE_{p,\tau}(\rho).
\]
So
\[
\Mod_{p,\si_\ep}(\Ga)=\cE_{p,\si_\ep}(\rho^*_\ep)=\cE_{p,\si}(\rho^*_\ep)+\ep\cE_{p,\tau}(\rho^*_\ep)\geq \Mod_{p,\si}(\Ga)+\ep\cE_{p,\tau}(\rho^*_\ep),
\]
and
\[
\Mod_{p,\si_\ep}(\Ga)\le\cE_{p,\si_\ep}(\rho^*_0)= \Mod_{p,\si}(\Ga)+\ep\cE_{p,\tau}(\rho^*_0),
\]
Thus
\[
\cE_{p,\tau}(\rho^*_\ep)\leq \frac{\phi(\si+\ep\tau)-\phi(\si)}{\ep}\leq \cE_{p,\tau}(\rho^*_0).
\]
But by part (2), $\rho^*_\ep\rightarrow\rho^*_0$, as $\ep\rightarrow 0$. So the directional derivative of $\phi$ in the direction of $\tau$ is :
\[
D_\tau(\phi)=\sum_{e\in E}\tau(e)\rho^*_{\si}(e)^p.
\]
Since part (2) then implies that all directional derivatives are continuous, it follows that $\phi$ is differentiable~\cite[Theorem 9.21]{baby-rudin}.

\end{proof}
\bibliographystyle{siamplain}
\bibliography{fulkerson-duality}
\def\cprime{$'$}

\end{document}